\documentclass[10pts]{amsart}

\usepackage{amssymb}
\usepackage{graphicx}
\usepackage{amsfonts}
\usepackage{amsmath}
\usepackage{amsthm}
\usepackage{fancyhdr}
\usepackage{indentfirst}
\usepackage{hyperref}
\usepackage{comment}
\usepackage{color}
\usepackage{subcaption}
\usepackage{epsfig}
\usepackage{pst-grad} 
\usepackage{pst-plot} 
\usepackage[space]{grffile} 
\usepackage{etoolbox} 
\usepackage{mathrsfs} 

\newtheorem{theorem}{Theorem}[section]
\newtheorem{lemma}[theorem]{Lemma}

\theoremstyle{definition}
\newtheorem{definition}[theorem]{Definition}

\newtheorem{proposition}[theorem]{Proposition}
\newtheorem{question}{Open Question}

\theoremstyle{remark}

\numberwithin{equation}{section}

\newcommand{\R}{\mathbb{R}}
\newcommand{\N}{\mathbb{N}}
\newcommand{\Z}{\mathbb{Z}}

\newcommand{\B}{\mathbb{B}}
\newcommand{\Sph}{\mathbb{S}}

\newcommand{\Div}{\operatorname{div}}
\newcommand{\Tr}{\operatorname{Tr}}

\newcommand{\ind}{\operatorname{index}}

\begin{document}

\title[Free boundary minimal surfaces in the unit ball]{Free boundary minimal surfaces in the unit ball : recent advances and open questions}
\date{\today}

\author[Martin Li]{Martin Man-chun Li}
\address{Department of Mathematics, The Chinese University of Hong Kong, Shatin, N.T., Hong Kong}
\email{martinli@math.cuhk.edu.hk}

\begin{abstract}
In this survey, we discuss some recent results on free boundary minimal surfaces in the Euclidean unit-ball. The subject has been a very active field of research in the past few years due to the seminal work of Fraser and Schoen on the extremal Steklov eigenvalue problem. We review several different techniques of constructing examples of embedded free boundary minimal surfaces in the unit ball. Next, we discuss some uniqueness results for free boundary minimal disks and the conjecture about the uniqueness of critical catenoid. We also discuss several Morse index estimates for free boundary minimal surfaces. Moreover, we describe estimates for the first Steklov eigenvalue on such free boundary minimal surfaces and various smooth compactness results. Finally, we mention some sharp area bounds for free boundary minimal submanifolds and related questions.
\end{abstract}

\maketitle

\section{Introduction}
\label{S:intro}

Minimal surfaces have been one of the most extensively studied objects in differential geometry, owing to its internal beauty and the connection with a wide range of mathematics including complex analysis, partial differential equations, conformal geometry, low dimensional topology and mathematical physics etc. The classical Plateau problem asks for an area-minimizing disk whose boundary spans a given Jordan curve in $\R^3$. Shortly after Douglas and Rad\'{o} independently gave a positive answer to this question, Courant \cite{Courant40} proposed to study the corresponding Neumann boundary value problem for minimal disks in $\R^3$. This problem, called the \emph{free boundary problem}, asks for an area-minimizing disk whose boundary is lying on a given constraint surface $S$ in $\R^3$. Courant and his students \cite{Courant} proved the first existence results of such minimizers under certain assumptions, for example that $S$ is not simply-connected. The boundary regularity was then actively studied by several mathematicians including Lewy \cite{Lewy51}, Hildebrandt \cite{Hildebrandt69} and Nitsche \cite{Hildebrandt-Nitsche79}, just to name a few. We refer the readers to a very well-written survey by Hildebrandt \cite{Hildebrandt86} which summarized some of the major development up to 1986.

Before the 1980s, most attention was drawn to study \emph{area-minimizers}. When the constraint surface $S$ is convex, it does not support any non-trivial area minimizers. It is therefore important to look for \emph{stationary} solutions, in other words general critical points to the area functional. In a pioneering paper \cite{Nitsche85}, Nitsche initiated the study of free boundary minimal (and constant mean curvature) surfaces in the unit ball $\B^3$ of $\R^3$. Several examples are given and a number of questions are formulated. Using methods from harmonic maps and geometric measure theory, a number of existence results (see for example \cite{Struwe84} \cite{Gruter-Jost86}) were established for \emph{unstable} free boundary minimal disks inside any convex body in $\R^3$. A general existence theory of free boundary minimal disks was later developed by Fraser \cite{Fraser00} which works in the setting of Riemannian manifolds of any dimension. Effective versions of Fraser’s result were recently established by Lin, Sun and Zhou \cite{LSZ} (see also Laurain and Petrides \cite{LP}).

Despite these tremendous successes, still not much is known after the work of Nitsche on free boundary minimal surfaces in $\B^3$. The situation changes drastically after the groundbreaking work of Fraser and Schoen \cite{Fraser-Schoen11} \cite{Fraser-Schoen16}. They found a close relationship between free boundary minimal surfaces in the unit $n$-ball $\B^n$ and an extremal eigenvalue problem on compact surfaces with boundary. They studied the spectrum, called \emph{Steklov eigenvalues}, of the Dirichlet-to-Neumann map on surfaces with boundary and discovered that any smooth metric maximizing the first normalized Steklov eigenvalue must be realized by a free boundary minimal surface in $\B^n$. This is analogous to the study of the extremal problem for the first eigenvalue of the Laplacian on closed surfaces, which is intimately related to closed minimal surfaces in $\Sph^n$ (see \cite{Fraser-Schoen13} for an excellent overview). More surprisingly, it was shown in \cite{Fraser-Schoen11} and \cite{Fraser-Schoen16} that there are many similarities between free boundary minimal surfaces in $\B^n$ and closed minimal surfaces in $\Sph^n$. Since then, there has been a lot of interest in the research community on free boundary minimal surfaces. In this survey, we will describe some of the recent advances focusing on the question of existence and uniqueness for free boundary minimal surfaces in $\B^n$, as well as some of their analytic and geometric properties.

We would like to point out that recently there has also been substantial progress in the general existence and regularity theory for free boundary minimal hypersurfaces in Riemannian manifolds with boundary. Thanks to the recent major breakthroughs in the Almgren-Pitts min-max theory by Marques and Neves, which have already led to the resolution of several longstanding conjectures in geometry and topology (see the ICM proceedings \cite{Marques-ICM} \cite{Neves-ICM} for a more detailed discussion), we have also witnessed rapid development in the free boundary setting. Some recent works related to this include \cite{Li15} \cite{Li-Zhou16} \cite{DeLellis-Ramic18} \cite{GWZ18} \cite{Guang-Wang-Zhou18} \cite{Wang} \cite{Wang19} \cite{GLWZ} \cite{W1} \cite{W2}.

\vspace{1ex}

\textbf{Acknowledgements.} The author is greatly indebted to Prof. Richard Schoen and Prof. Ailana Fraser for their continuous encouragement and introducing him to the beautiful world of free boundary minimal surfaces. He is also grateful to Prof. Shing-Tung Yau for the valuable opportunity to deliver an invited lecture at the seventh International Congress of Chinese Mathematicians in Beijing and the first annual meeting of International Consortium of Chinese Mathematicians in Guangzhou. The author would also like to thank Antonio Ros, Romain Petrides and Renato Bettiol for their interest of this work and pointing out several references. The author is substantially supported by a research grant from the Research Grants Council of the Hong Kong Special Administrative Region, China [Project No.: CUHK 24305115] and CUHK Direct Grant [Project Code: 4053291].

\section{Definitions and preliminaries}
\label{S:Prelim}

In this section, we will introduce some definitions and notations which will be used throughout the rest of this paper. We also establish some foundational properties for free boundary minimal submanifolds in $\B^n$.

\subsection{Definitions}

Throughout this paper, we denote, for any $n \geq 3$, the closed $n$-dimensional Euclidean unit-ball by
\[ \B^n:=\{(x_1,x_2,\cdots,x_n)  \in \R^n: x_1^2 + x_2^2 + \cdots + x_n^2 \leq 1\}.\]
The boundary $\partial \B^n$ is the $(n-1)$-dimensional unit sphere denoted by $\Sph^{n-1}$.

Let $\Sigma$ be a $k$-dimensional smooth manifold with boundary $\partial \Sigma$. Consider a smooth immersion $\varphi:\Sigma \to \B^n$ such that $\varphi(\partial \Sigma) \subset \partial \B^n$, the extrinsic curvature of $\Sigma$ is described by a vector-valued symmetric two-tensor $A:T \Sigma \otimes T\Sigma \to N \Sigma$, called the \emph{second fundamental form} of $\Sigma$, defined by
\[ A(u,v):=(D_u v)^N, \]
where $D$ is the Euclidean connection on $\R^n$ and $(\cdot)^N$ is the normal component of a vector with respect to $\Sigma$. Here, $T\Sigma$ and $N\Sigma$ are respectively the tangent and normal bundle of the immersed submanifold $\Sigma$ in $\R^n$. The \emph{mean curvature vector} $\vec{H}$ of $\Sigma$ is then defined to be the trace of its second fundamental form, i.e. at each $p \in \Sigma$,
\[ \vec{H}(p):= \Tr A(p)=\sum_{i=1}^k (D_{e_i} e_i)^N(p) \]
where $\{e_1,\cdots,e_k\}$ is an orthonormal basis of the tangent space $T_p \Sigma$. The \emph{outward conormal} $\nu$ of $\partial \Sigma$ is defined such that at each $p \in \partial \Sigma$, $\nu(p)$ is the unique unit vector in $T_p\Sigma$ which is orthogonal to $T_p\partial \Sigma$ and pointing out of $\Sigma$.

\begin{definition}
An immersed (resp. embedded) submanifold $\varphi:\Sigma^k \to \B^n$ with $\varphi(\partial \Sigma) \subset \partial \B^n$ is said to be an immersed (resp. embedded) \emph{free boundary minimal submanifold} if both of the following holds:
\begin{itemize}
\item[(i)] $\Sigma$ is a minimal submanifold (i.e. $\vec{H} \equiv 0$);
\item[(ii)] $\Sigma$ meets $\partial \B^n$ orthogonally along $\partial \Sigma$ (i.e. $\nu \perp \partial \B^n$).
\end{itemize}
We often call $\Sigma$ a \emph{free boundary minimal surface} when $k=2$ and a \emph{free boundary minimal hypersurface} when $k=n-1$.
\end{definition}

Free boundary minimal submanifolds can be characterized variationally as critical points of the area functional as follows. Suppose we have a smooth one-parameter family of immersions $\varphi_t:\Sigma \to \B^n$ with $\varphi_t(\partial \Sigma) \subset \partial \B^n$ for all $t \in (-\epsilon,\epsilon)$. The first variational formula (see e.g. \cite[\S 1.3]{CM}) gives
\[ \delta \Sigma (X):=\left. \frac{d}{dt}\right|_{t=0} \textrm{Area}(\varphi_t(\Sigma))=- \int_\Sigma \langle \vec{H}, X \rangle \;da  + \int_{\partial \Sigma} \langle \nu, X \rangle \; ds,\]
where $X=\frac{\partial}{\partial t} \varphi_t$ is the variational vector field; $da$ and $ds$ are the $k$-dimensional and $(k-1)$-dimensional measures induced by $\varphi$ on $\Sigma$ and $\partial \Sigma$ respectively. From the above formula it is clear that $\Sigma$ is a free boundary minimal submanifold if and only if $\delta \Sigma(X)=0$ for all variational vector field $X$ along $\Sigma$

Note that the above discussions hold with $\B^n$ replaced by any $n$-dimensional Riemannian manifold with boundary. Nonetheless, when the ambient space is $\B^n$, there is another useful characterization of free boundary minimal submanifolds. Recall that (see e.g. \cite[Proposition 1.7]{CM}) a submanifold $\Sigma \subset \R^n$ is minimal if and only if the coordinate functions $x_i$, $i=1,\cdots,n$, of $\R^n$ restrict to harmonic functions on $\Sigma$ (with respect to the intrinsic Laplacian $\Delta_\Sigma$). On the other hand, since the outward unit normal of $\partial \B^n$ is given by the position vector. It is easy to see that the free boundary condition $\nu \perp \partial \B^n$ is equivalent to the condition that $\partial x_i/\partial \nu=x_i$ for $i=1,\cdots,n$.

We summarize our discussions above into the following theorem.

\begin{theorem}
\label{T:definitions}
Let $\varphi:\Sigma^k \to \B^n$ be an immersed $k$-dimensional submanifold with $\varphi (\partial \Sigma) \subset \partial \B^n$. Then the following statements are equivalent:
\begin{itemize}
\item[(1)] $\Sigma$ is a free boundary minimal submanifold.
\item[(2)] $\Sigma$ is a critical point of the area functional among all $k$-dimensional submanifolds in $\B^n$ with boundary lying on $\partial \B^n$.
\item[(3)] The coordinate functions $x_i$ of $\R^n$ restrict to harmonic functions on $\Sigma$ and they satisfy the Robin boundary condition $\partial x_i/\partial \nu=x_i$ along $\partial \Sigma$.
\end{itemize}
\end{theorem}

\subsection{Basic properties}

Before we proceed, we derive several geometric properties of free boundary minimal submanifolds in $\B^n$. In the remaining part of this section, we denote $\Sigma$ to be a $k$-dimensional immersed free boundary minimal submanifold in $\B^n$. 

First of all, we show that $\partial \Sigma$ must be non-empty and $\Sigma$ cannot touch $\partial \B^n$ at an interior point. Both of these are consequences of the maximum principle.

\begin{proposition}[Properness]
All free boundary minimal submanifolds $\Sigma$ in $\B^n$ are proper, i.e. $\varphi(\Sigma) \cap \partial \B^n=\varphi(\partial \Sigma) \neq \emptyset$.
\end{proposition}

\begin{proof}
Since the coordinate functions $x_i$ are harmonic functions on $\Sigma$ (with respect to the induced metric), if we let $|x|^2=\sum_{i=1}^n x_i^2$, we have
\begin{equation}
\label{E:Laplace-r^2}
\Delta_\Sigma |x|^2 = 2 \sum_{i=1}^n x_i \Delta_\Sigma x_i +2 \sum_{i=1}^n |\nabla^\Sigma x_i|^2= 2k.
\end{equation}
Therefore, $|x|^2$ is a strictly sub-harmonic function on $\Sigma$ and hence the maximum can only be achieved on $\partial \Sigma$. This implies the proposition.
\end{proof}

Next, we give a relationship between the volume of $\Sigma$ and the boundary volume of $\partial \Sigma$.

\begin{proposition}
\label{P:length-area}
$k|\Sigma|=|\partial \Sigma|$.
\end{proposition}

\begin{proof}
We integrate (\ref{E:Laplace-r^2}) over $\Sigma$ on both sides. Applying integration by parts and using the fact that $\partial x_i/\partial \nu=x_i$, $i=1,\cdots,n$, on $\partial \Sigma$, 
\[ 2|\partial \Sigma|=2 \int_{\partial \Sigma} \sum_{i=1}^n x_i^2 \; ds =\int_{\partial \Sigma} \frac{\partial}{\partial \nu} |x|^2 \; ds =\int_\Sigma \Delta_\Sigma |x|^2 \; da = 2k |\Sigma|. \]
In the first equality we have used that $\sum_{i=1}^n x_i^2=1$ on $\partial \Sigma$ since $\varphi(\partial \Sigma) \subset \partial \B^n$.
\end{proof}

The next proposition shows that the boundary $\partial \Sigma$ must be \emph{balanced}, i.e. the center of mass of $\partial \Sigma$ is the origin.

\begin{proposition}
$\int_{\partial \Sigma} x_i \; ds=0$ for each $i=1,\cdots,n$.
\end{proposition}

\begin{proof}
Note that each $x_i$ is harmonic in $\Sigma$ and satisfies $\partial x_i/\partial \nu=x_i$ along $\partial \Sigma$, integrating by parts gives
\[ \int_{\partial \Sigma} x_i \; ds =\int_{\partial \Sigma} \frac{\partial x_i}{\partial \nu} \; ds = \int_\Sigma \Delta_\Sigma x_i \; da =0.\]
\end{proof}

Since $\pi_1(\B^n)=0$, we have the following well-known fact from topology.

\begin{proposition}
If $\Sigma^k \subset \B^n$ is an embedded hypersurface (i.e. $k=n-1$), then $\Sigma$ is orientable.
\end{proposition}

\section{Existence results}
\label{S:Existence}

While there are no closed minimal submanifolds in $\R^n$, there do exist some interesting non-trivial examples of free boundary minimal submanifolds in $\B^n$. It is a difficult problem to construct embedded examples of free boundary minimal submanifolds in $\B^n$. For a long time the only known examples in $\B^3$ were given by the equatorial disk and the critical catenoid \cite{Nitsche85}. The central question is:

\begin{question}
Are there any topological obstructions for a $k$-dimension submanifold $\Sigma$ to be a free boundary minimal submanifold in $\B^n$? Does it depend on whether $\Sigma$ is immersed or embedded?
\end{question}

Despite many new existence results in the past few years, we still do not have a satisfactory answer to \textbf{Open Question 1} even in the case $k=2=n-1$. It is interesting to compare it with the analogous question for closed minimal submanifolds in $\Sph^n$. When $n=3$, this analogous question was solved completely by Lawson \cite{Lawson70}, who proved that every closed surface (orientable or not) except $\R \mathbb{P}^2$ can be minimally \emph{embedded} into $\Sph^3$. To indicate the subtlety of the question in the free boundary setting, it is yet unknown whether there exists a free boundary minimal surface in $\B^3$ with genus $1$ and one boundary component.

\subsection{Low cohomogeneity examples}

The simplest examples of free boundary minimal submanifolds are the \emph{equatorial $k$-disks} $\B^k \subset \B^n$ given by the intersection of $\B^n$ with any $k$-dimensional subspace of $\R^n$. These examples are flat and totally geodesic. When $k=2$, we will sometimes denote it by $\mathbb{D}$.

To look for the next simplest examples, it is natural to impose maximal symmetries. The first example of such is described by Nitsche \cite{Nitsche85} for the case $k=2=n-1$. Consider the surface of revolution (about the $x_3$-axis) given by the equation
\[ \sqrt{x_1^2 + x_2^2} = \alpha \cosh \left( \frac{x_3}{\alpha} \right)\]
which describes a classical complete embedded minimal surface in $\R^3$ called the \emph{catenoid} (of certain scale described by the parameter $\alpha>0$). By a simple direct calculation, one can show that the catenoid defined above intersects $\partial \B^n$ orthogonally if and only if $\alpha=(\beta^2 +\cosh^2 \beta)^{-1/2}$ where $\beta>0$ is the unique positive solution to the (transcendental) equation
\[ \beta=\coth \beta.\]
Numerically, we have $\beta \approx 1.19968$ and $\alpha \approx 0.460485$. The restriction of such a catenoid in $\B^3$ yields an embedded rotationally symmetric free boundary minimal annulus called the \emph{critical catenoid} \cite{Fraser-Schoen11}. 

\begin{figure}[h]
    \centering
    \begin{subfigure}{.43\textwidth}
        \centering
        \includegraphics[width=1\textwidth]{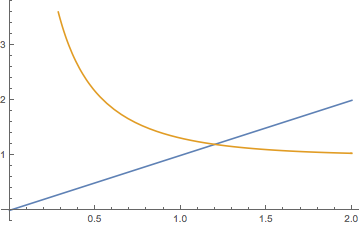}
    \end{subfigure}
    \begin{subfigure}{.43\textwidth}
        \centering
        \includegraphics[width=1\textwidth]{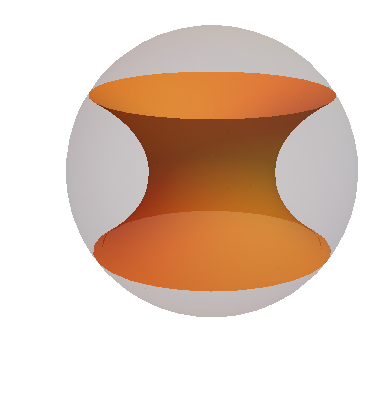}
    \end{subfigure}
    \caption{The critical catenoid (on the left is a plot of the curves $y=x$ and $y=\coth x$ which intersect at $\beta \approx 1.19968$)}
    \label{Pic:catenoid}
\end{figure}

In higher dimensions there also exists a complete embedded $O(n-1)$-symmetric minimal hypersurface in $\R^n$ (see e.g. \cite{Tam-Zhou09}). After possibly a rotation and translation, we can assume that such a hypersurface is rotationally invariant with respect to the $x_n$-axis. As before, up to the scale $\alpha>0$ the complete $(n-1)$-catenoid is given by
\[ \sqrt{x_1^2 + x_2^2 + \cdots +x_n^2}= \alpha f \left( \frac{x_n}{\alpha} \right) \]
where $f:(-L_n,L_n) \to (0,\infty)$ is the unique maximally defined positive solution to the second order ordinary differential equation $f''=(n-2) f^{2n-5}$ with $f(0)=1$ and $f'(0)=0$. Again, a direct computation shows that such a minimal hypersurface intesects $\partial \B^n$ orthogonally if and only if $\alpha=(\beta^2+f(\beta)^2)^{-1/2}$ where $\beta \in (-L_n,L_n)$ is the unique positive solution to the equation (see \cite[\S 2.1]{Smith-Stern-Tran-Zhou17} for more details)
\[ \beta f'(\beta)=f(\beta).\]
Therefore, we arrive at the following:
\begin{theorem}[Nitsche \cite{Nitsche85}, Fraser-Schoen \cite{Fraser-Schoen11}, Smith-Stern-Tran-Zhou \cite{Smith-Stern-Tran-Zhou17}]
\label{T:catenoid-exists}
For each $n \geq 3$, there exists an embedded rotationally symmetric (i.e. $O(n-1)$-invariant) free boundary minimal hypersurface $\mathbb{K}^{n-1}$ in $\B^n$, each of which is homeomorphic to the cylinder $\Sph^{n-2} \times [-1,1]$. We call $\mathbb{K}^{n-1}$ the \emph{$(n-1)$-dimensional critical catenoid}. When $n=3$, we simply call it the \emph{critical catenoid} denoted by $\mathbb{K}$.
\end{theorem}

It is not hard to see from an ODE uniqueness argument that the equatorial disks and the critical catenoids exhaust all the cohomogeneity one examples. Some other low cohomogeneity examples were considered by Q.-M. Wang in \cite{Wa}. In \cite{Freidin-Gulian-McGrath17}, Freidin, Gulian and McGrath considered cohomogeneity two examples of free boundary minimal hypersurfaces in $\B^n$ for $n \geq 4$ which are $O(k_1)\times O(k_2)$-invariant with $k_1+k_2=n$. By analyzing the ODEs corresponding to the free boundary minimal surface equation, they proved the following:

\begin{theorem}[Freidin-Gulian-McGrath \cite{Freidin-Gulian-McGrath17}]
\label{T:low-cohomo}
Let $k_1,k_2>1$ be any positive integers such that $k_1+k_2=n$. 
\begin{itemize}
\item[(1)] For $n<8$, there exists an infinite family $\{\Sigma^{FGM}_{k_1,k_2}(\ell)\}_{\ell \in \N}$ of mutually non-congruent, embedded, $O(k_1) \times O(k_2)$-invariant free boundary minimal hypersurfaces in $\B^n$, each of which is homeomorphic to $\B^{k_1} \times \Sph^{k_2-1}$. 
\item[(2)] For $n \geq 8$, there exists an embedded $O(k_1)\times O(k_2)$-invariant free boundary minimal hypersurface $\Sigma^{FGM}_{k_1,k_2}$ in $\B^n$, each of which is homeomorphic to $\Sph^{k_1-1} \times \Sph^{k_2-1} \times [-1,1]$.
\end{itemize}
\end{theorem} 

For case (1) above, as $\ell \to \infty$, the family $\Sigma^{FGM}_{k_1,k_2}(\ell)$ converges in the Hausdorff sense (and smoothly away from the origin) to the minimal cone over the Clifford minimal hypersurface $\Sph^{k_1-1}\left( \sqrt{\frac{k_1-1}{n-2}}\right) \times \Sph^{k_2-1}\left( \sqrt{\frac{k_2-1}{n-2}} \right)$ in $\Sph^{n-1}$. The examples above are obtained by solving the ODEs with suitable boundary conditions after imposing the $O(k_1) \times O(k_2)$ symmetry (similar methods were employed by Hsiang and Lawson \cite{Hsiang-Lawson71} to construct closed minimal submanifolds in $\Sph^n$). Note that there is an intriguing dichotomy feature in Theorem \ref{T:low-cohomo} depending on the dimension $n$, which is related to the existence of non-flat minimal cones in higher dimensions \cite{Simons68}.

\subsection{Extremal eigenvalue problem}

As mentioned in the introduction, the subject of free boundary minimal surfaces has attracted a lot of recent attention mostly due to the groundbreaking work of Fraser and Schoen \cite{Fraser-Schoen11} \cite{Fraser-Schoen16} on the extremal Steklov eigenvalue problem on compact surfaces with boundary. We now give a brief discussion on their important work.

Let $\Sigma$ be a smooth compact manifold with boundary. Given a Riemannian metric $g$ on $\Sigma$, one can define the \emph{Dirichlet-to-Neumann map} $L:C^\infty(\partial \Sigma) \to C^\infty(\partial \Sigma)$ by 
\[ L u:=\frac{\partial \hat{u}}{\partial \nu} \]
where $\nu$ is the outward unit normal of $\partial \Sigma$ with respect to $(\Sigma,g)$ and $\hat{u} \in C^\infty(\Sigma)$ is the harmonic extension of $u$, i.e. $\hat{u}$ is the unique solution to the Dirichlet boundary value problem
\[ \left\{ \begin{array}{rcll}
\Delta_g \hat{u} & = & 0 & \text{ on $\Sigma$}\\
\hat{u}|_{\partial \Sigma} & = & u.  & 
\end{array} \right. \]
It is well-known that $L$ is a non-negative, self-adjoint elliptic pseudodifferential operator of order one which has a discrete spectrum
\[ 0=\sigma_0 < \sigma_1 \leq \sigma_2 \leq \cdots \leq \sigma_k \leq \cdots \to +\infty.\]
These are called the \emph{Steklov eigenvalues} of $(\Sigma,g)$. The eigenfunctions $\{\phi_i\}_{i=0}^\infty \subset C^\infty(\partial \Sigma)$ form a complete orthonormal basis of $L^2(\partial \Sigma)$. For our convenience, we will sometimes think of $\phi_i$ as defined on the whole $\Sigma$ by their harmonic extension $\hat{\phi}_i$. The lowest eigenvalue $\sigma_0$ is always zero which corresponds to constant functions. Using Theorem \ref{T:definitions} (3), Fraser and Schoen made an important observation which links free boundary minimal submanifolds in $\B^n$ to the Steklov eigenvalues.

\begin{lemma}[Fraser-Schoen \cite{Fraser-Schoen11}]
An immersed submanifold $\varphi:\Sigma^k \to \B^n$ is a free boundary minimal submanifold if and only if the coordinate functions $x_i$, $i=1,\cdots,n$, restricted to $\Sigma$ are Steklov eigenfunctions with eigenvalue $1$. 
\end{lemma}

The first non-zero Steklov eigenvalue $\sigma_1$ can be characterized variationally as
\begin{equation}
\label{E:sigma1}
\sigma_1=\inf_{0 \neq u \in C^\infty(\partial \Sigma)} \left\{ \frac{\int_\Sigma |\nabla \hat{u}|^2 \; da}{\int_{\partial \Sigma} u^2 \; ds} \; : \; \int_{\partial \Sigma} u\; ds =0 \right\}. 
\end{equation}
More generally, we have for any $k \in \N$,
\[ \sigma_k=\inf_{0 \neq u \in C^\infty(\partial \Sigma)} \left\{ \frac{\int_\Sigma |\nabla \hat{u}|^2 \; da}{\int_{\partial \Sigma} u^2 \; ds} \; : \;  \int_{\partial \Sigma} u \phi_j \; ds =0 \text{ for $j=0,1,2,\cdots,k-1$}\right\}, \]
where $\phi_j \in C^\infty(\partial \Sigma)$ is the $j$-th Steklov eigenfunction corresponding to the $j$-th Steklov eigenvalue $\sigma_j$. 

We now focus on the case that $\Sigma$ is a surface. By the work of Fraser and Schoen \cite{Fraser-Schoen11} (for $k=1$) and Girouard and Polterovich \cite{Girouard-Polterovich12} (for $k \geq 2$), we have the following coarse upper bound for the $k$-th Steklov eigenvalue only in terms of the topology of the surface. Note that the left hand side in the inequality is a scale-invariant quantity. One can equivalently consider only smooth metrics $g$ normalized so that $|\partial \Sigma|_g=1$.

\begin{proposition}[Fraser-Schoen \cite{Fraser-Schoen11}, Girouard-Polterovich \cite{Girouard-Polterovich12}]
\label{P:coarse-bound}
Let $\Sigma$ be a smooth compact surface with genus $\gamma$ and $\ell$ boundary components. Then, for all $k \in \N$,
\[ \sup_g \sigma_k(g) |\partial \Sigma|_g \leq 2 \pi (\gamma+\ell)k \]
where the supremum is taken over all smooth Riemnanian metrics $g$ on $\Sigma$. Here, $|\partial \Sigma|_g$ denotes the total boundary length of $\partial \Sigma$ with respect to the metric $g$.
\end{proposition}

\begin{proof}
We recall the proof of Fraser and Schoen in the case $k=1$. Their idea is to use the variational characterization of $\sigma_1(\Sigma)$ together with Hersch's balancing trick to construct suitable test functions. By a result of Ahlfors and Gabard, there exists a proper conformal branched cover $\varphi:\Sigma \to D$ to the unit disk $D$ in $\R^2$. Moreover, the degree of $\varphi$ is at most $\gamma+\ell$. After possibly composing $\varphi$ with a suitable conformal diffeomorphism of $D$, we can assume that the covering map $\varphi=(\varphi_1,\varphi_2)$ is balanced on the boundary, i.e. $\int_{\partial \Sigma} \varphi_i=0$ for $i=1,2$. Using (\ref{E:sigma1}), we obtain
\[ \sigma_1 \int_{\partial \Sigma} \varphi_i^2 \; ds \leq \int_\Sigma |\nabla \hat{\varphi_i}|^2 \; da \leq \int_\Sigma |\nabla \varphi_i|^2\;da \]
where we have used the fact that harmonic functions minimize the Dirichlet energy integral with fixed boundary value. Summing $i=1,2$, note that $\varphi(\partial \Sigma) \subset \partial D$ and using conformality of $\varphi$, we have
\[ \sigma_1 |\partial \Sigma|_g \leq 2 \text{Area}(\varphi(\Sigma)) \leq 2\pi (\gamma+\ell).\]
\end{proof}

We remark that the upper bound in Proposition \ref{P:coarse-bound} can be improved to a linear bound of the form $A\gamma +Bk$ \cite[Corollary 4.1]{Hassannezhad11} or $2\pi(\gamma+\ell+k-2)$ \cite[Theorem 1.4]{Karpukhin17}. The precise form does not concern us here as long as it is solely depending on $k$ and the topology of $\Sigma$. Given the coarse upper bound in Proposition \ref{P:coarse-bound}, it is reasonable to study the following extremal eigenvalue problems:

\begin{question}[Fraser-Schoen \cite{Fraser-Schoen16}]
For each topological type of $\Sigma$, what is the sharp upper bound for $\sigma_k(g) |\partial \Sigma|_g$? Is it achieved by a smooth metric on $\Sigma$?
\end{question}

For $k=1$ and simply connected domains in $\R^2$, this problem was studied in 1954 by Weinstock \cite{Weinstock54} where he showed that the bound in Proposition \ref{P:coarse-bound} is in fact sharp and equality is only achieved by the round disk. More than fifty years later, Fraser and Schoen initiated the study again on the annulus (also for $k=1$) and found in \cite{Fraser-Schoen11} \cite{Fraser-Schoen16} that the sharp upper bound is achieved by the rotationally invariant metric on the critical catenoid $\mathbb{K}$ defined in Theorem \ref{T:catenoid-exists}. Remarkably, they were also able to solve the problem completely for any genus $0$ surfaces with boundary. The problem is highly non-trivial even for the annulus case as apriori there is no guarantee that a \emph{smooth} extremal metric should exist (in fact, sometimes an extremal metric ought to be singular \cite{Girouard-Polterovich10}). However, if we \emph{assume} that a smooth extremal metric exists, then there is a nice geometric characterization of the metric arising from free boundary minimal surfaces in $\B^n$.

\begin{theorem}[Fraser-Schoen \cite{Fraser-Schoen16}]
\label{T:regularity}
Let $\Sigma$ be a compact surface with boundary. If $g_0$ is a smooth metric on $\Sigma$ such that for some $k \in \N$,
\[ \sigma_k(g_0) |\partial \Sigma|_{g_0}=\max_g \sigma_k(g) |\partial \Sigma|_g\]
where the maximum is taken over all smooth metrics on $\Sigma$. Then there exist independent first Steklov eigenfunctions $u_1,\cdots,u_n$ which give a branched conformal immersion $u=(u_1,\cdots,u_n):\Sigma \to \B^n$ such that $u(\Sigma)$ is a free boundary minimal surface in $\B^n$, and up to rescaling of the metric $u$ is an isometry on $\partial \Sigma$.
\end{theorem}

Theorem \ref{T:regularity} reduces the extremal eigenvalue problem to studying the regularity of an extremal metric. The regularity is a very subtle issue. In \cite{Fraser-Schoen16}, Fraser and Schoen completely solved the problem for $k=1$ on all genus zero surfaces. Their arguments are very clever and sophisticated. Roughly speaking, their proof involves first controlling the conformal structure of metrics near the supremum, and then controlling the metrics themselves. A key ingredient in the first part is to establish, in the genus zero case, the strict monotonicity of the supremum of $\sigma_1 |\partial \Sigma|$ with respect to the number of boundary components. Moreover, they were able to give a precise control on the multiplicity of $\sigma_1$ and a detailed description of the geometry of the extremal metric in the genus $0$ case. As a corollary of their main theorem \cite[Theorem 1.1]{Fraser-Schoen16}, we have the following important existence result. (See also the recent work of Girouard and Lagace \cite{GL}.)

\begin{theorem}[Fraser-Schoen \cite{Fraser-Schoen16}]
\label{T:FS-genus-0}
For any $\ell \in \N$, there exists a smooth embedded free boundary minimal surface $\Sigma^{FS}_{\ell}$ in $\B^3$ of genus $0$ and $\ell$ boundary components. Each $\Sigma^{FS}_\ell$ is star-shaped in the sense that a ray from the origin hits $\Sigma^{FS}_\ell$ at most once. Moreover, after a suitable rotation of each $\Sigma^{FS}_\ell$, the sequence $\{\Sigma^{FS}_\ell\}_{\ell \in \N}$ converges in $C^3$-norm on compact subsets of the interior of $\B^3$ to the equatorial disk with multiplicity two.
\end{theorem}

Note that up to rotations, we have $\Sigma^{FS}_1=\mathbb{D}$ and $\Sigma^{FS}_2=\mathbb{K}$. A schematic picture of the Fraser-Schoen surfaces are shown in Figure \ref{Fig:FS}.

\begin{figure}[h]
    \centering
    \begin{subfigure}{.3\textwidth}
        \centering
        \includegraphics[width=1\textwidth]{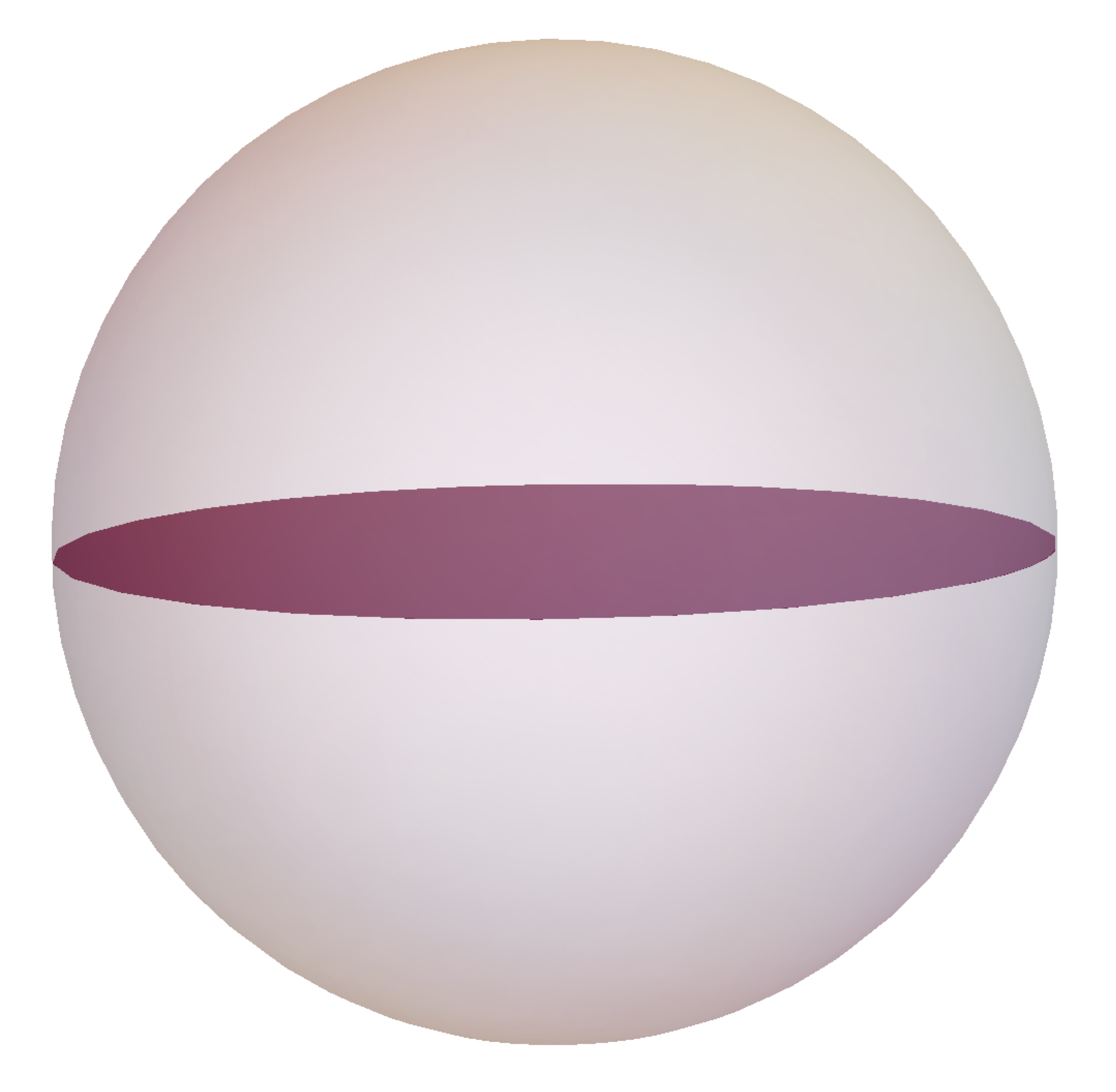}
    \end{subfigure}
    \begin{subfigure}{.3\textwidth}
        \centering
        \includegraphics[width=1\textwidth]{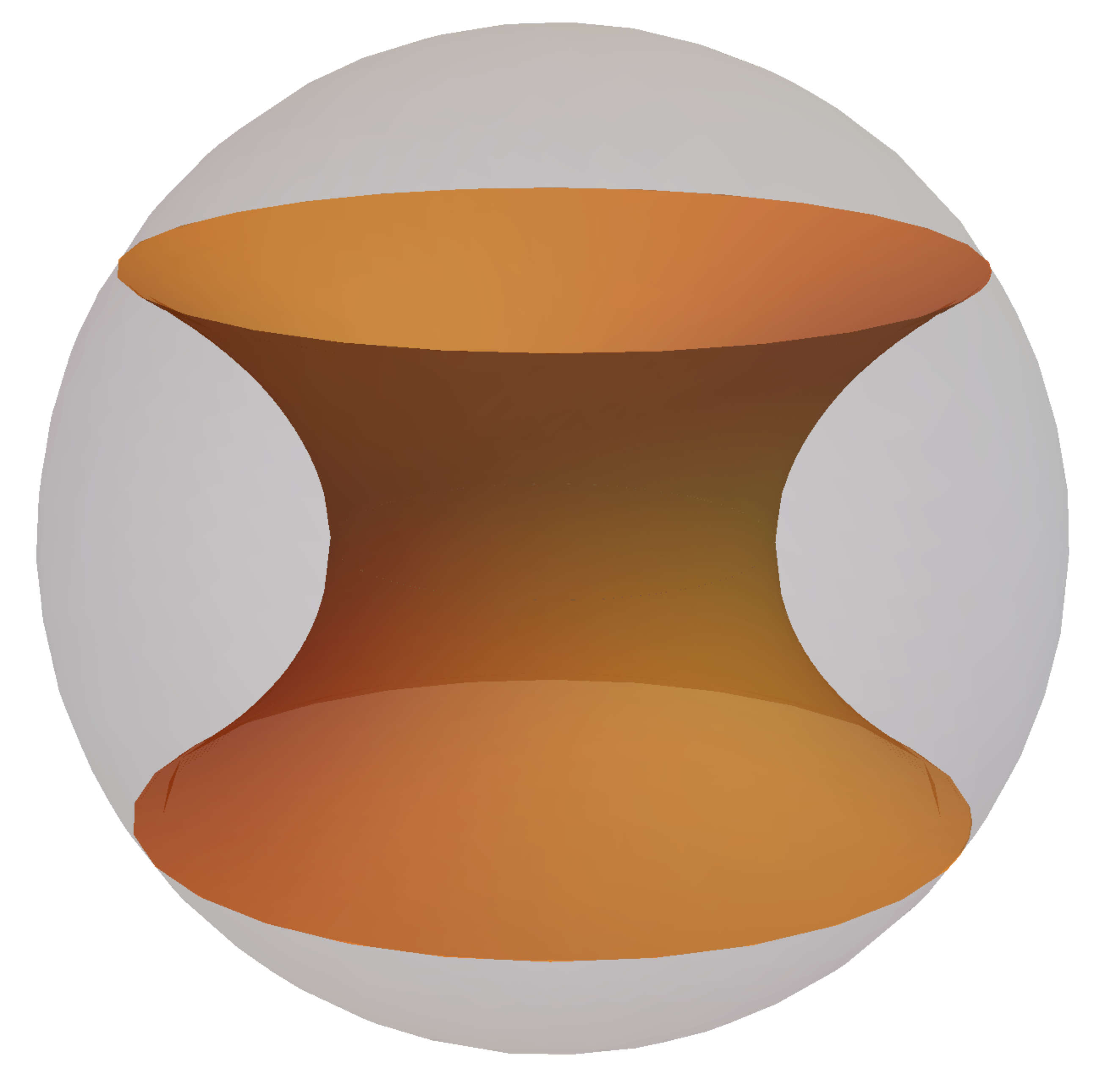}
    \end{subfigure}
       \begin{subfigure}{.3\textwidth}
        \centering
        \includegraphics[width=1\textwidth]{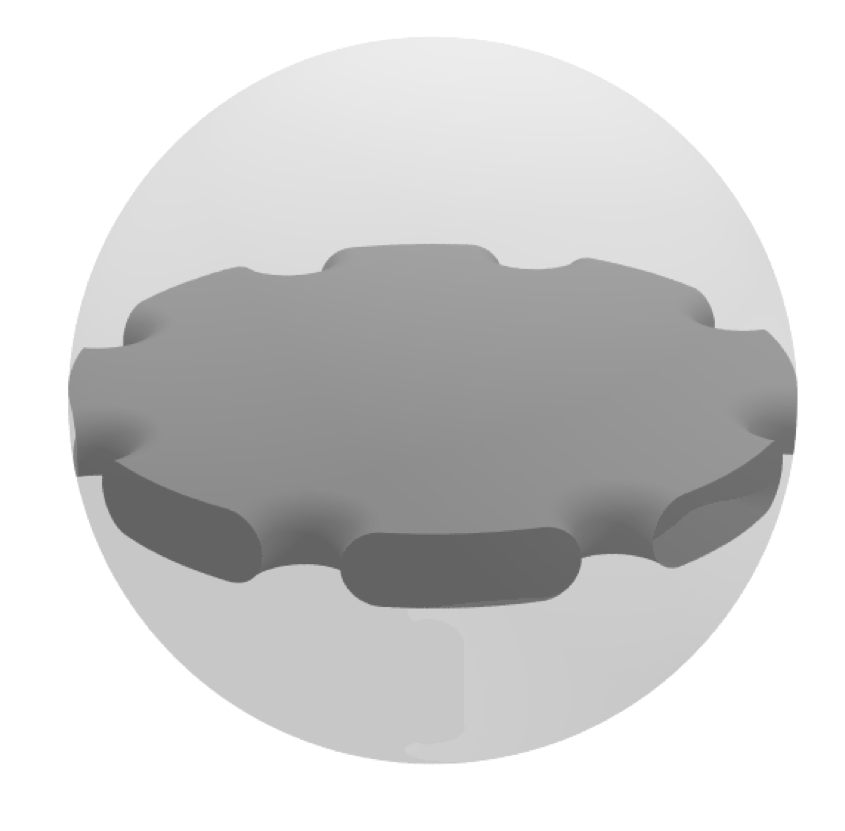}
    \end{subfigure}
    \caption{Fraser-Schoen surfaces}
    \label{Fig:FS}
\end{figure}

In the same paper \cite[Theorem 1.5]{Fraser-Schoen16}, they were also able to extend their arguments (for $k=1$) to the case of M\"{o}bius band. Note that the extremal metric is embedded as a free boundary minimal surface in $\B^4$ instead of $\B^3$.

\begin{theorem}[Fraser-Schoen \cite{Fraser-Schoen16}]
There exists a smooth embedded free boundary minimal surface in $\B^4$ called the \emph{critical M\"{o}bius band} given by a suitable rescaling of the embedding $\varphi:[-T_0,T_0] \times \mathbb{S}^1/\sim \to \R^4$ given by 
\[ \varphi (t,\theta)=(2 \sinh t  \cos \theta, 2 \sinh t \sin \theta, \cosh 2t \cos 2\theta, \cosh 2t \sin 2\theta) \]
where $T_0>0$ is the unique positive solution of $\coth t = 2 \tanh 2 t$. Here, we think of the M\"{o}bius band as $[-T_0,T_0] \times \mathbb{S}^1$ under the identification $(t,\theta) \sim (-t, \theta+\pi)$.
\end{theorem}

We remark that there are recent progress towards the remaining cases of positive genus for \textbf{Open Question 2} in the work of \cite{Pe, Kar, HPe}. Their results provide new examples of positive genus free boundary minimal surfaces in $\mathbb{B}^N$ for some $N \geq 3$. It is, however, unknown whether $N=3$ nor the surfaces are embedded.


Not much is known concerning the extremal eigenvalue problem for higher Steklov eigenvalues $\sigma_k$ when $k \geq 2$. However, we would like to mention an interesting result of Fan, Tam and Yu \cite{Fan-Tam-Yu15} who studied the extremal eigenvalue problem for $\sigma_k$, $k \geq 2$, on the space of smooth rotationally invariant metrics on annuli. They found that except for $k=2$, the supremum (among rotationally invariant metrics) of $\sigma_k |\partial \Sigma|$ is achieved by a smooth extremal metric. For $k=2m-1$, $m \in \N$, the extremal metric is achieved by the $m$-fold cover (thus immersed when $m>1$) of the critical catenoid $\mathbb{K}$ in $\B^3$. This does not produce any new examples of free boundary minimal surfaces in $\B^3$. However, when $k$ is even, this produces some new immersed examples in $\B^4$ other than the critical M\"{o}bius band.

\begin{theorem}[Fan-Tam-Tu \cite{Fan-Tam-Yu15}]
For each $m \in \N$, $m \geq 2$, there exists a smooth immersed free boundary minimal surface in $\B^4$ called the \emph{critical $m$-M\"{o}bius band} given by a suitable rescaling of the immersion $\varphi_m:[-T_m,T_m] \times \mathbb{S}^1/\sim \to \R^4$ defined by 
\[ \varphi_m (t,\theta)=(m \sinh t  \cos \theta, m \sinh t \sin \theta, \cosh mt \cos m\theta, \cosh mt \sin m\theta) \]
where $T_m>0$ is the unique positive solution of $\coth t = m \tanh m t$. Here, we think of the M\"{o}bius band as $[-T_m,T_m] \times \mathbb{S}^1$ under the identification $(t,\theta) \sim (-t, \theta+\pi)$.
\end{theorem}

Note that \cite{Fan-Tam-Yu15} only proved the examples above achieve the supremum among all \emph{rotationally invariant} metrics on the annulus. It is interesting to see whether they also maximize among \emph{all} smooth metrics on the annulus without the rotational symmetry assumption. It is also conjectured that the supremum of $\sigma_2 |\partial \Sigma|$ is not achieved by any smooth metric (c.f. \cite{Girouard-Polterovich10}). Note that Fraser and Sargent \cite{FSa} recently obtained some new existence and classification results for $\mathbb{S}^1$-invariant free boundary minimal annuli and M\"{o}bius bands.

\subsection{Gluing techniques}

Another powerful tool in constructing examples of natural variationally defined geometric objects is the \emph{singular perturbation method}. The idea originated from the work of Schoen \cite{Schoen88} on constant scalar curvature metrics and was later pioneered by a series of important work of Kapouleas \cite{Kapouleas90} \cite{Kapouleas91} \cite{Kapouleas95} \cite{Kapouleas97} on the construction of complete minimal and constant mean curvature surfaces. In the past decade, we have witnessed tremendous success of gluing techniques in a wide range of geometric variational problems, including for example closed minimal surfaces in $\Sph^n$ \cite{Kapouleas-Yang10} \cite{Kapouleas17} \cite{Kapouleas-McGrath17}, complete constant mean curvature surfaces in $\R^3$ \cite{Breiner-Kapouleas14}, Special Lagrangian cones \cite{Haskins-Kapouleas07} \cite{Haskins-Kapouleas12} \cite{Haskins-Kapouleas13} and self shrinkers of mean curvature flow \cite{Kapouleas-Kleene-Moller18}.

In the past few years, many of these powerful techniques have been applied to construct many new examples of embedded free boundary minimal surfaces in $\B^3$. The constructions can be divided into two different types: \emph{desingularization} and \emph{doubling} (or \emph{tripling}) constructions. For a desingularization construction, one typically takes a number of known free boundary minimal surfaces which intersect each other transversely along some curves. Then one hopes to desingularize the curves of intersection to produce another free boundary minimal surface. For a doubling (or tripling) construction, one takes two (or three) nearby copies of a given free boundary minimal surface. Then we try to produce another free boundary minimal surface by connecting the nearby copies using small catenoidal necks (in the interior) or half-catenoidal bridges (near the boundary). In both of these constructions, a family of initial surfaces have to be built carefully out of known examples so that they are approximate solutions to the free boundary minimal surface equation. One then applies an implicit function theorem argument to show that \emph{one} of these initial surfaces can be perturbed to an exact solution. We refer the readers to the excellent surveys \cite{Kapouleas05} \cite{Kapouleas11} by Kapouleas on the basic methodology of these ideas.

The first desingularization construction of free boundary minimal surfaces in $\B^3$ was done by Kapouleas and the author in \cite{Kapouleas-Li17}. This is a highly symmetric construction with $O(2)$-symmetry. The idea is to take an equatorial disk $\mathbb{D}$ and a critical catenoid $\mathbb{K}$ which are rotationally invariant about the same axis. Their union $\mathbb{D} \cup \mathbb{K}$ is a free boundary minimal surface in $\B^3$ which is singular along the circle of intersection $\mathcal{C}:=\mathbb{D} \cap \mathbb{K}$. We proved that this configuration can be desingularized to produce a free boundary analogue of the Costa-Hoffman-Meeks surfaces which are complete embedded minimal surfaces in $\R^3$ with finite total curvature. 

\begin{figure}[h]
    \centering
    \begin{subfigure}{.43\textwidth}
        \centering
        \includegraphics[width=1\textwidth]{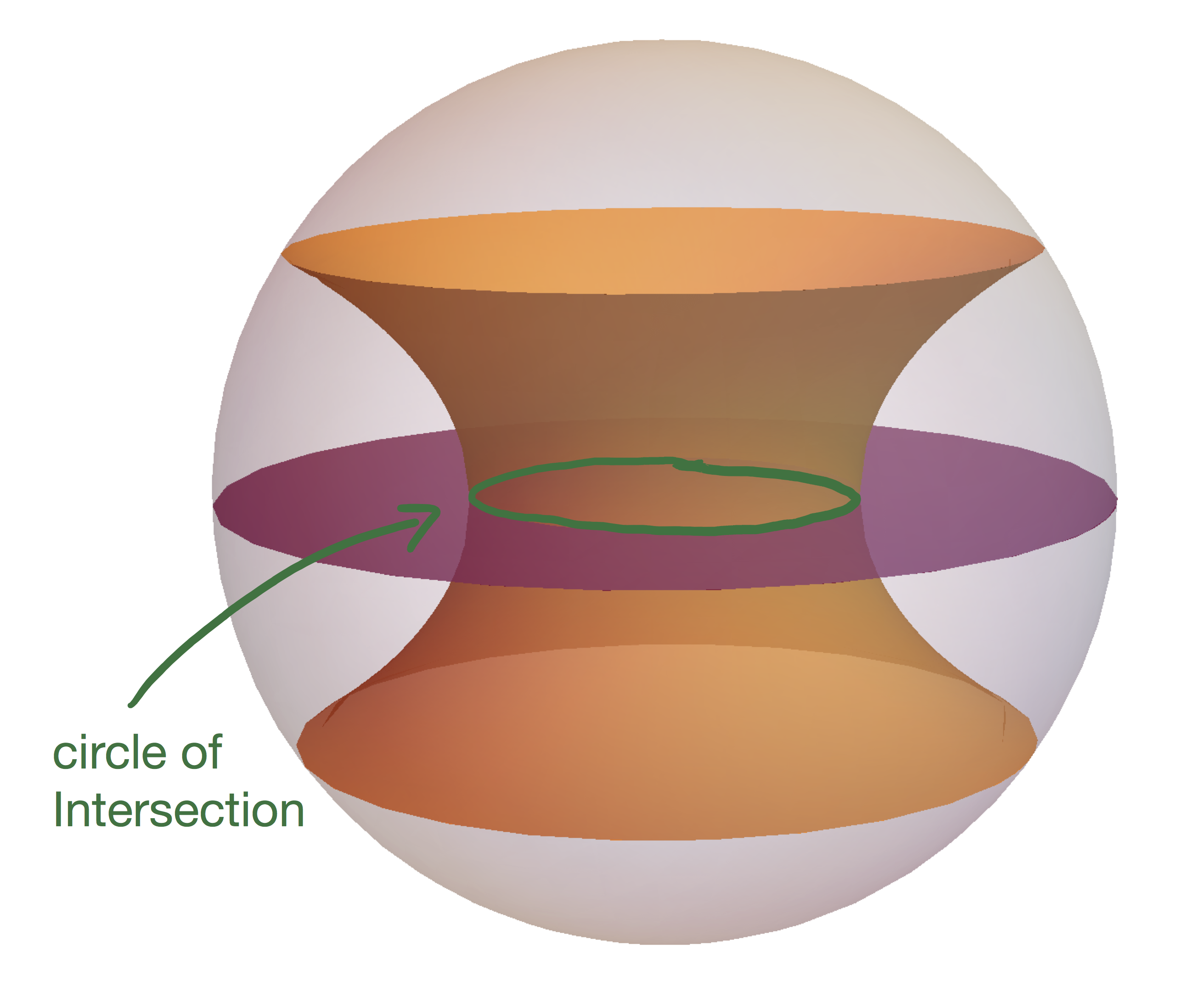}
    \end{subfigure}
    \begin{subfigure}{.5\textwidth}
        \centering
        \includegraphics[width=1\textwidth]{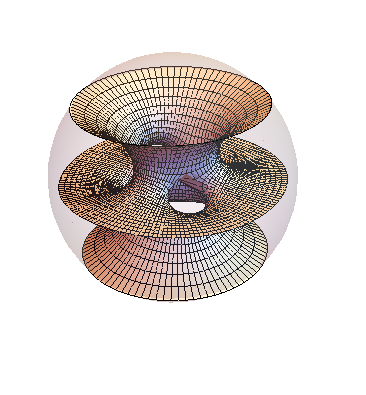}
    \end{subfigure}
    \caption{Desingularizing the union of a pair of coaxial equatorial disk and critical catenoid to obtain a free boundary analogue of Costa-Hoffman-Meeks surfaces}
    \label{Fig:KL}
\end{figure}

To construct our initial surfaces, one needs to take away a tubular neighhood of the circle of singularity $\mathcal{C}$ and glue in a smooth \emph{desingularizing surface} which is approximately minimal (in suitable scale). The desingularizing surface is constructed out of the standard \emph{Scherk's surface} $\mathcal{S}$ which is a complete singly-periodic embedded minimal surface in $\R^3$ described by the equation
\[ \sinh x_1 \sinh x_2 = \cos x_3.\]
The Scherk's surface $\mathcal{S}$ provides a model to desingularize the line of intersection between two mutually orthogonal planes in $\R^3$. Since the singularity $\mathcal{C}$ at hand is a circle instead of a line, an extra bending (which creates the majority of non-zero mean curvature other than the transition regions within where the pieces are glued) has to be introduced to fit $\mathcal{S}$ into a neighborhood of $\mathcal{C}$. Owing to the $O(2)$-symmetry, the scale can be chosen to be uniform along $\mathcal{C}$ and is measured by $m^{-1}$ where $m$ is the number of handles of the bent Scherk surface. By carefully analyzing the linearization of the free boundary minimal surface equation, taking the symmetries into account, we found that there is a one-dimensional kernel of the linearized equation on the building blocks. By the geometric principle, one can account for this one-dimensional kernel by considering a one-parameter family of initial surfaces $M_{\theta,m}$, where $\theta$ measures the amount of rotation on the ``wings'' of the Scherk surface to create \emph{unbalancing} to cancel the one dimensional kernel. Combining all the uniform linear and non-linear estimates, one can show using Schauder's fixed point theorem that for all $m$ sufficiently large, one of the initial surfaces $M_{\theta,m}$ can be perturbed to an exact solution of the non-linear free boundary minimal surface equation. As a result, we have the following:

\begin{theorem}[Kapouleas-Li \cite{Kapouleas-Li17}]
\label{T:KL}
For any $g \in \N$ sufficiently large, there exists an embedded, orientable, smooth free boundary minimal surface $\Sigma_g^{KL}$ in $\B^3$ which has genus $g$ and three connected boundary components. The surface $\Sigma_g^{KL}$ is symmetric under a dihedral group of order $4g+4$. Moreover, as $g \to \infty$, the sequence $\{\Sigma_g^{KL}\}$ converges in the Hausdorff sense (and smooth away from $\mathcal{C}=\mathbb{D} \cap \mathbb{K}$) to $\mathbb{D} \cup \mathbb{K}$.
\end{theorem}

Note that the surfaces $\Sigma_g^{KL}$ have the same symmetry as the Costa-Hoffman-Meeks surfaces which are complete properly embedded minimal surfaces in $\R^3$. Finally we mention that in an article under preparation, we construct free boundary minimal surfaces of arbitrary high genus with connected boundary, and also ones with two boundary components, by desingularizing two disks intersecting orthogonally along a diameter of the unit three-ball. The intersecting disks configuration is clearly not rotationally invariant, and the symmetry group is small and independent of the genus. These features make such desingularization construction much harder, but we can overcome the difficulties by following the general approach in \cite{Kapouleas11} with appropriate modifications to handle the free boundary condition. This construction can also be extended to the case of more than two disks symmetrically arranged around a common diameter by using higher order Karcher-Scherk towers as models. See Figure \ref{Fig:KL2} for examples of such setup.

\begin{figure}[h]
    \centering
    \begin{subfigure}{.43\textwidth}
        \centering
        \includegraphics[width=1\textwidth]{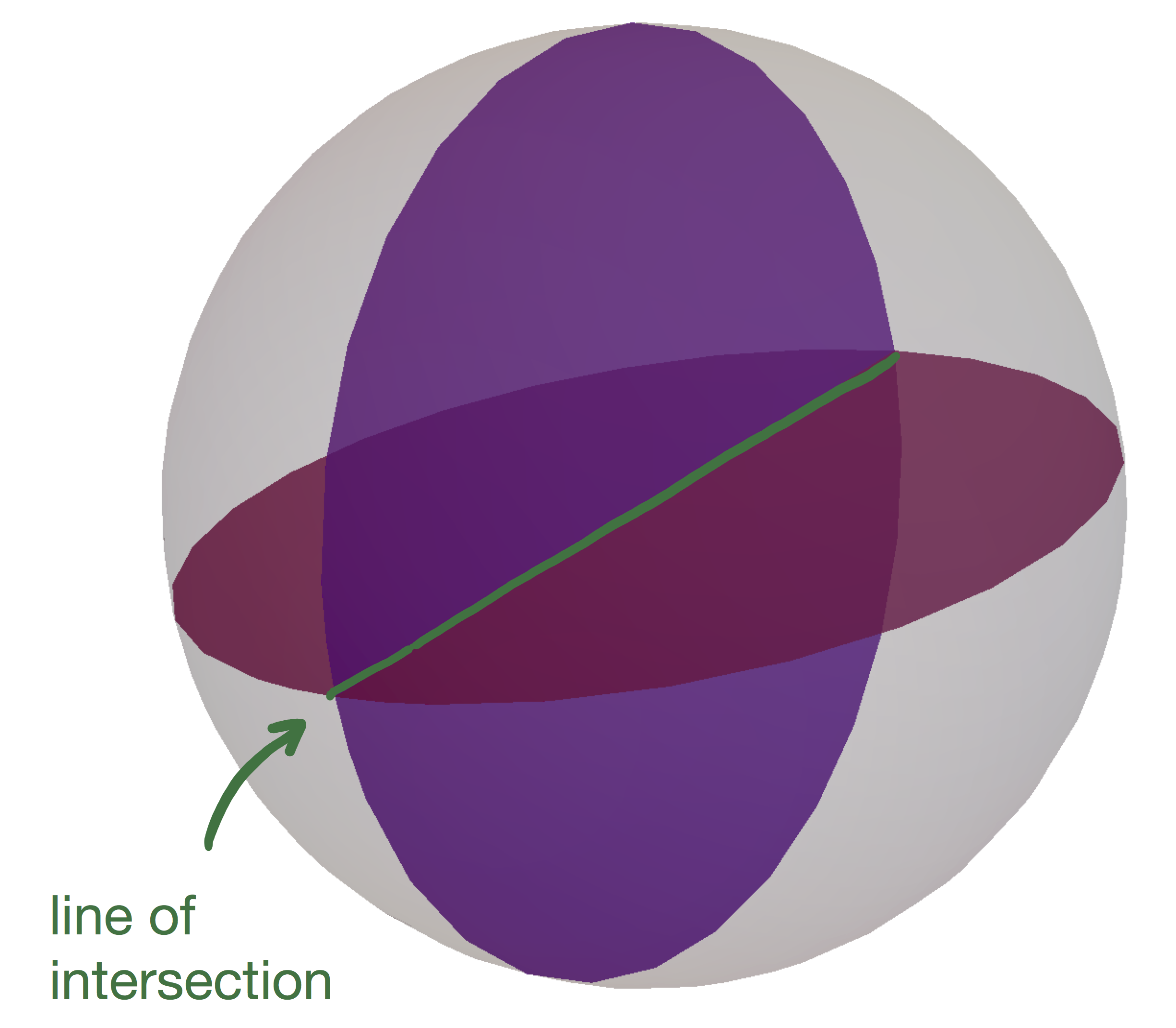}
    \end{subfigure}
    \begin{subfigure}{.43\textwidth}
        \centering
        \includegraphics[width=1\textwidth]{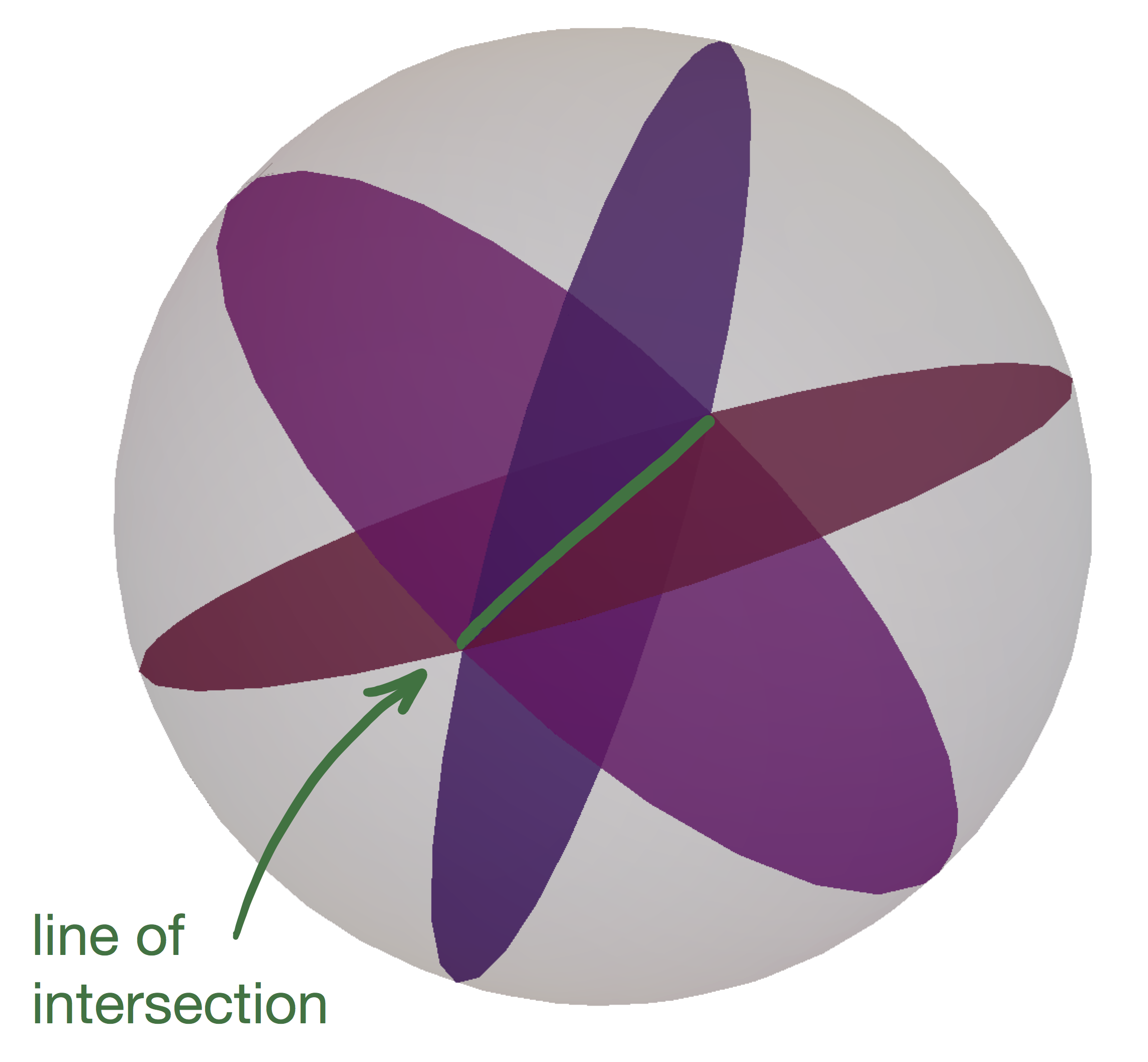}
    \end{subfigure}
    \caption{The intersection of two and three equatorial disks which possess much less symmetry.}
    \label{Fig:KL2}
\end{figure}

On the other hand, there have been several doubling and tripling constructions of free boundary minimal surfaces in $\B^3$. The first such construction is given by Folha, Pacard and Zolotareva in \cite{Folha-Pacard-Zolotareva17}. They considered taking two nearby copies of the equatorial disk $\mathbb{D}$ joined together by a large number of half catenoidal bridges near the boundary (and possibly a catenoidal neck joining the center of the disks) and showed that one of these can be perturbed to an exact free boundary minimal surface in $\B^3$. As a result, they obtained the following:

\begin{theorem}[Folha-Pacard-Zolotareva \cite{Folha-Pacard-Zolotareva17}]
For any $n \in \N$ sufficiently large, there exists a genus $0$ surface $\Sigma_n^{FPZ}$ and a genus $1$ surface $\tilde{\Sigma}_n^{FPZ}$ which are both smooth embedded free boundary minimal surfaces in $\B^3$ with $n$ boundary components. As $n \to \infty$, each of the sequences $\{\Sigma_n^{FPZ}\}$ and $\{\tilde{\Sigma}_n^{FPZ}\}$ converges in the Hausdorff sense (and smooth away from the catenoidal neck and bridges) to an equatorial disk with multiplicity two.
\end{theorem}

Each of the surfaces $\Sigma_n^{FPZ}$ and $\tilde{\Sigma}_n^{FPZ}$ is symmetric about a discrete group of reflections generated by planes passing through the rotationally invariant axis and a horizontal reflection switching the two copies of the disks. Note that the surfaces $\Sigma^{FPZ}_n$ are homeomorphic to the Fraser-Schoen surfaces $\Sigma^{FS}_n$ in Theorem \ref{T:FS-genus-0}. It is conjectured that they are actually the same surfaces . Note that some positive evidence was given by McGrath \cite[Corollary 2]{McGrath18} who showed that $\sigma_1(\Sigma^{FPZ}_n)=1$. 

In a recent paper of Kapouleas and Wiygul \cite{Kapouleas-Wiygul17}, they considered a tripling construction from the equatorial disks and produced surprisingly an infinite family of new examples of free boundary minimal surfaces in $\B^3$ with \emph{connected} boundary. 

\begin{theorem}[Kapouleas-Wiygul \cite{Kapouleas-Wiygul17}]
For any $g \in \N$ sufficiently large, there exists a smooth embedded free boundary minimal surface $\Sigma_g^{KW}$ in $\B^3$ which has genus $g$ and one boundary component. As $g \to \infty$, the sequence $\{\Sigma_g^{KW}\}$ converges in the Hausdorff sense (and smooth away from the catenoidal bridges) to the equatorial disk $\mathbb{D}$ with multiplicity three.
\end{theorem}

Note that the methodology in \cite{Kapouleas-Wiygul17} can also be applied to the situation of stacking more than three copies of the equatorial disk as done in \cite{Wiygul15} for the Clifford torus in $\mathbb{S}^3$. Kapouleas and McGrath \cite{Kapouleas-McGrath} have recently announced a doubling construction of the critical catenoid which would produce new examples of high genus and 4 boundary components.

\subsection{Min-max construction} 

During the last few years there has been substantial development on the existence theory of minimal surfaces. In particular, the min-max method pioneered by Almgren and Pitts \cite{Pitts} has been greatly advanced by a series of recent work of Marques and Neves (see e.g. \cite{Marques-ICM} \cite{Neves-ICM} for an excellent survey of some recent results). In the free boundary setting, Gr\"{u}ter and Jost \cite{Gruter-Jost86} were the first to apply the continuous min-max construction of Smith and Simon \cite{Smith82} to produce an embedded free boundary minimal disk in any convex body of $\R^3$ (there was also another approach of harmonic maps by Struwe \cite{Struwe84} using the method of Sacks and Uhlenbeck). A more general existence result was obtained by the author in \cite{Li15}. In a joint work with X. Zhou, the author established the Almgren-Pitts min-max theory in the free boundary setting for any compact Riemannian manifold with boundary. Combining with Lusternik-Schnirelmann theory, we were able to show that there exist infinitely many smooth, geometrically distinct, properly embedded free boundary minimal hypersurfaces in any compact Riemannian manifold $(M^n,g)$, where $3 \leq n \leq 7$, with nonnegative Ricci curvature and strictly convex boundary. More generally, the same holds for any Riemannian manifold $(M,g)$ satisfying the \emph{embedded Frankel property} (c.f. Theorem \ref{T:Frankel} and \cite[Definition 1.3]{Marques-Neves17}). 

Unfortunately, the theorem mentioned above does not say anything new when the ambient space is $\B^n$. The reason is that once an embedded free boundary minimal hypersurface (say $\B^{n-1}$) is given, we can generate an infinite family of examples by the isometries of $\B^n$. This corresponds to the degenerate case of Lusternik-Schnirelmann theory that $\omega_p=\omega_{p+1}$ for some $p \in \N$ where $\omega_p$ here is the $p$-width of the ambient manifold. In fact, by Morse index considerations (c.f. Theorem \ref{T:FS-index} and \cite{Marques-Neves16}) it is not hard to see that $\omega_1(\B^n)=\omega_2(\B^n)= \cdots=\omega_n(\B^n)=|\B^{n-1}|$.

We now give a brief discussion on the \emph{$p$-width} of $\B^n$. In what follows we will use some terminology from geometric measure theory, we refer the readers to \cite{Li-Zhou16} for more details and precise definitions. Let $\mathcal{Z}_{n-1,rel}(\B^n,\partial \B^n;\Z_2)$ be the space of equivalence classes of $(n-1)$-dimensional relative cycles in $\B^n$ with $\Z_2$ coefficients. Equipped with the flat topology, the space $\mathcal{Z}_{n-1,rel}(\B^n,\partial \B^n;\Z_2)$ is weakly homotopically equivalent to $\R\mathbb{P}^\infty$. Hence, $H^1(\mathcal{Z}_{n-1,rel}(\B^n,\partial \B^n;\Z_2),\Z_2)=\Z_2$ and we denote its generator by $\overline{\lambda}$. For any finite dimensional simplicial complex $X$, a map $\Phi:X \to \mathcal{Z}_{n-1,rel}(\B^n,\partial \B^n;\Z_2)$ is called a \emph{$p$-sweepout} if $\Phi$ is continuous in the flat topology and $\Phi^*(\overline{\lambda}^p) \neq 0 \in H^p(X;\Z_2)$, where $\overline{\lambda}^p=\overline{\lambda} \cup \cdots \cup \overline{\lambda}$ is the $k$-times cup product of $\overline{\lambda}$ in $H^*(\mathcal{Z}_{n-1,rel}(\B^n,\partial \B^n;\Z_2),\Z_2)$. Denote $\mathcal{P}_p$ to be the set of all $p$-sweepouts $\Phi$ that have \emph{no concentration of mass} (see \cite[\S 3.3]{Marques-Neves17} for a precise definition). We define the \emph{$p$-width} of $\B^n$ as
\[ \omega_p(\B^n):=\inf_{\Phi \in \mathcal{P}_p} \sup \{ |\Phi(x)| \; : \; x \in \textrm{dmn}(\Phi)\}, \]
where $\textrm{dmn}(\Phi)$ is the domain of definition of $\Phi$. As explained in \cite{Marques-Neves17}, the $p$-width can be expressed as the infimum of the widths of homotopy classes of discrete $p$-sweepouts (see \cite{Marques-Neves17} for definitions):
\[ \omega_p(\B^n)=\inf_{\Pi \in \mathcal{D}_p} \textbf{L}(\Pi).\]
By the min-max theorem of \cite{Li-Zhou16} and the embedded Frankel property (Theorem \ref{T:Frankel}), for each $p \in \N$, there exists a smooth embedded free boundary minimal hypersurface $\Sigma_p$ in $\B^n$ and a positive integer $n_p \in \N$ such that 
\begin{equation}
\label{E:min-max}
\omega_p(\B^n)=n_p |\Sigma_p|.
\end{equation}
In \cite{Liokumovich-Marques-Neves18}, Lioukumovich, Marques and Neves proved that the $p$-widths satisfy a Weyl law conjectured by Gromov:
\[ \lim_{p \to \infty} \omega_p(\B^n) p^{-1/n}= a(n) |\B^n|^{\frac{n-1}{n}},\]
where $a(n)>0$ is a universal constant depending only on $n$. Using a sharp area bound for free boundary minimal hypersurfaces in $\B^n$ (see Theorem \ref{T:Brendle}), this implies that the multiplicity $n_p$ in (\ref{E:min-max}) can grow at most like $O(p^{1/n})$ as $p \to \infty$. 

\begin{question}
Is $\{\Sigma_p\}_{p \in \N}$ an infinite set? If so, does $|\Sigma_p| \to +\infty$ as $p \to \infty$?
\end{question}

It is generally believed that the minimal hypersurfaces realizing the $p$-width would become dense and equidistributed in the ambient space. Making use of the Weyl law \cite{Liokumovich-Marques-Neves18}, this has been proved for generic metrics in closed Riemannian manifolds by Irie-Marques-Neves \cite{Irie-Marques-Neves18} and Marques-Neves-Song \cite{Marques-Neves-Song17}. We would also like to mention that in a very recent paper \cite{Song18}, Song settled in its full generality (i.e. without assuming genericity of the metric) Yau's conjecture on the existence of infinitely many closed minimal hypersurfaces in any closed Riemannian manifold. Note that there are some very recent progress in the free boundary setting made by Z. Wang \cite{W1, W2}.

On the other hand, using a variant of the min-max theory in the continuous and equivariant setting, Ketover constructed another infinite family of examples of embedded free boundary minimal surfaces in $\B^3$.

\begin{theorem}[Ketover \cite{Ketover16}]
For each $g \geq 1$, there exists a smooth embedded free boundary minimal surface $\Sigma^K_g$ in $\B^3$ which is symmetric under a dihedral group of order $4g+4$. When $g$ is sufficiently large, $\Sigma^K_g$ has genus $g$ and three boundary components. Moreover, as $g \to \infty$, the sequence $\{\Sigma^K_g\}$ converges in the varifold sense to $\mathbb{D} \cup \mathbb{K}$. 
\end{theorem} 

The proof was inspired by a plausible variational construction of the Costa-Hoffman-Meeks surfaces suggested by Pitts and Rubinstein. By imposing the dihedral group symmetry, Ketover applied the equivariant min-max theory \cite{Ketover16b} to produce the surfaces $\Sigma^K_g$. It is interesting to note that one is only able to get a precise control on the topology of $\Sigma^K_g$ for \emph{large enough} $g$. Given that the topology and the symmetry group are the same as those constructed in Theorem \ref{T:KL}, it is reasonable to ask the following:

\begin{question}
For all $g$ sufficiently large, are the surfaces $\Sigma^{KL}_g$ and $\Sigma^K_g$ congruent to each other?
\end{question}

Ketover also proved that surfaces of genus $0$ with $\ell$ boundary components, isotopic to the Fraser-Schoen surfaces $\Sigma^{FS}_\ell$ in Theorem \ref{T:FS-genus-0} can be constructed variationally from a one-parameter equivariant min-max procedure \cite[Theorem 1.2]{Ketover16}. Finally, using the symmetry group of the Platonic solids, he was able to find three possibly new genus zero examples \cite[Theorem 1.3]{Ketover16}. These are the free boundary analogues of the construction of closed minimal surfaces in $\Sph^3$ by Karcher, Pinkall and Sterling \cite{Karcher-Pinkall-Sterling88}.

\begin{theorem}[Ketover \cite{Ketover16}]
There exists an example of free boundary minimal surface in $\B^3$ with tetrahedral symmetry of genus $0$ with $4$ boundary components; an example with octahedral symmetry of genus $0$ with $6$ boundary components; and an example of dodocahedral symmetry of genus $0$ with $12$ boundary components.
\end{theorem}

It is not clear that whether it is possible that these surfaces are the same as the Fraser-Schoen surfaces $\Sigma^{FS}_\ell$ for $\ell=4,6,12$. We would like to remark that Carlotto, Franz and Schulz has recently announced in \cite{CFS} some constructions of new free boundary minimal surfaces in $\mathbb{B}^3$ with connected boundary and arbitrary genus using min-max methods.

\section{Uniqueness questions for free boundary minimal surfaces}
\label{S:Uniqueness}

In this section, we discuss some uniqueness results for free boundary minimal surfaces in $\B^n$ which are either topologically a disk or an annulus. Since any isometry of $\B^n$ (e.g. rotations and reflections but not translations) takes a free boundary minimal surface to another free boundary minimal surface, one can at most determine a free boundary minimal surface up to isometries of $\B^n$.

\subsection{Disk-type solutions}

In 1985, Nitsche proved the following uniqueness theorem for disk-type free boundary minimal surfaces in $\B^3$. Note that the result holds even for \emph{immersed} surfaces.

\begin{theorem}[Nitsche \cite{Nitsche85}]
\label{T:Nitsche}
The equatorial disk is the only (up to rigid motions of $\B^3$) immersed free boundary minimal disk in $\B^3$. 
\end{theorem}

\begin{proof}
The proof is based on a Hopf-differential argument. Using isothermal coordinates, we can represent the minimal disk as a conformal minimal immersion $u:D \to \B^3$. Let $h$ be its second fundamental form extended $\mathbb{C}$-linearly and consider the quadratic differential $\Phi:=h(\frac{\partial}{\partial z},\frac{\partial}{\partial z})dz^2$ defined on $D$. Minimality implies that $\Phi$ is holomorphic and the free boundary condition implies that $\Phi$ is totally real on $\partial D$. By Riemann-Roch, no such non-trivial holomorphic quadratic differential exists on $D$ and hence $\Phi \equiv 0$. Thus, $u$ is totally geodesic and hence must be an equatorial disk.
\end{proof}

Recently, Fraser and Schoen generalized Nitsche's result to higher codimensions and allowing possibly some branch points. This is surprising as there are many non-totally geodesic immersed minimal two-spheres in $\mathbb{S}^n$ for $n \geq 4$ \cite{Calabi67}. This shows that free boundary minimal surfaces in $\B^n$ are in some sense more rigid than closed minimal surfaces in $\Sph^n$. 

\begin{theorem}[Fraser-Schoen \cite{Fraser-Schoen15}]
Any proper branched, immersed free boundary minimal disk in $\B^n$, $n \geq 3$, must be an equatorial plane disk.
\end{theorem}

\begin{proof}
As in the proof of Theorem \ref{T:Nitsche}, we can take a conformal branched minimal immersion $u:D \to \B^n$ (note that the normal bundle is smooth across branch points). Consider the complex vector-valued function $\Phi:=z^4 (u_{zz}^\perp \cdot u_{zz}^\perp)$ where $u_{zz}^\perp$ denotes the component of $u_{zz}$ orthogonal to $\Sigma=u(D)$ and $\cdot$ denotes the $\mathbb{C}$-bilinear product, one can show using minimality that $\Phi$ is holomorphic on $D$ and the free boundary condition that $\Phi$ is totally real on $\partial D$. Therefore, $\Phi$ must be constant but $\Phi(0)=0$. Thus, we have $\Phi \equiv 0$ which implies that the second fundamental form of $\Sigma$ is zero on $\partial \Sigma$. Hence, $\partial \Sigma$ must be a great circle and $\Sigma$ must be an equatorial plane disk.
\end{proof}

We would like to point out that the arguments of Fraser and Schoen also work for free boundary disks with parallel mean curvature vector inside a geodesic ball of any $n$-dimensional space form.

\subsection{Free boundary minimal annuli}

In \cite{Nitsche85}, it was claimed without proof that the critical catenoid is the only free boundary minimal annulus in $\B^3$. The conjecture was restated again in \cite{Fraser-Li14}.

\begin{question}
Is the critical catenoid the only (up to rigid motions in $\B^3$) embedded free boundary minimal annulus in $\B^3$?
\end{question}

One should compare this conjecture with the well-known Lawson's conjecture which states that the only embedded minimal torus in $\Sph^3$ is the Clifford torus. The Lawson conjecture was answered affirmatively by Brendle \cite{Brendle13} using a maximum principle argument on an ingenious choice of two-point function. We expect that \emph{embeddedness} is an essential assumption as in the case of Lawson conjecture (see, for example, \cite[Theorem 1.4]{Brendle-survey}). However, we would like to point out that at the moment there is no single example (other than an $n$-fold covering of the critical catenoid) of an immersed free boundary minimal annulus in $\B^3$ which is not embedded. 

The following result may be helpful towards an answer to \textbf{Open Question 5}. 

\begin{lemma}
An immersed free boundary minimal annulus in $\B^3$ has no umbilic points. Hence, the second fundamental form is nowhere vanishing on the surface.
\end{lemma}

\begin{proof}
Fix a conformal parametrization $u:D(r)\setminus D(1) \to \B^3$ of the minimal annulus where $r>1$ is determined by its conformal type. Consider the complexified second fundamental form $h$ as in the proof of Theorem \ref{T:Nitsche}. We see that by the same argument that $h(\frac{\partial}{\partial z},\frac{\partial}{\partial z})$ is holomorphic inside the annulus and totally real on the two boundary circles. By countably many reflections we have a bounded holomorphic function on $\mathbb{C} \setminus \{0\}$, which by a removable singularity theorem can be extended to all of $\mathbb{C}$. The Liouville theorem then implies that it must be constant. The constant cannot be zero, otherwise the surface is totally geodesic. Therefore, $h$ is non-vanishing everywhere so there are no umbilic points on the surface.
\end{proof}

As in the proof of Lawson's conjecture, the key obstacle to \textbf{Open Question 5} is to exploit the \emph{embeddedness} of the minimal surface. It is unclear whether Brendle's proof of the Lawson conjecture can be adapted to this setting to answer \textbf{Open Question 5}. However, by a Bj\"{o}rling-type uniqueness result for free boundary minimal surfaces \cite[Corollary 3.9]{Kapouleas-Li17}, it suffices to show that one of the boundary components of the minimal annulus is rotationally symmetric.

Other characterizations of the critcal catenoid have been found. The following deep uniqueness result is obtained by Fraser and Schoen \cite[Theorem 1.2]{Fraser-Schoen16}. Their proof uses several ingredients including an analysis of the second variation of energy and area for free boundary surfaces. 

\begin{theorem}[Fraser-Schoen \cite{Fraser-Schoen16}]
\label{T:FS-unique}
Let $\Sigma$ be a free boundary minimal annulus in $\B^n$ immersed by the first Steklov eigenfunctions. Then $n=3$ and $\Sigma$ is congruent to the critical catenoid.
\end{theorem}

On the other hand, a partial result towards \textbf{Open Question 5} was obtained by McGrath \cite{McGrath18} saying that under additional symmetry assumptions, the critical catenoid is the unique embedded free boundary minimal annulus in $\B^3$.

\begin{theorem}[McGrath \cite{McGrath18}]
Let $\Sigma \subset \B^3$ be an embedded free boundary minimal annulus which is symmetric with respect to the coordinate planes. Then, up to rigid motion in $\B^3$, $\Sigma$ is the critical catenoid.
\end{theorem}

The proof by McGrath uses in many places the nodal domain theorem for Steklov eigenfunctions \cite{Kuttler-Sigillito69}. Using a symmetrization argument of Choe and Soret \cite{Choe-Soret09}, he proved that certain symmetry assumptions would imply $\sigma_1=1$. Combining this with the uniqueness theorem (Theorem \ref{T:FS-unique}) of Fraser-Schoen, the result follows.

\section{Morse index estimates}
\label{S:Morse}

In this section, we study the second variation of area for a free boundary minimal submanifold in $\B^n$. 

\subsection{Morse index bounds}

Recall that a free boundary minimal submanifold $\Sigma^k$ in $\B^n$ is a critical point to the area functional among the class of submanifolds with boundary lying on $\partial \B^n$. If one considers the second variation of area with respect to normal variations (see \cite{Schoen06} for example), we have the second variation formula
\[ \delta^2 \Sigma (W,W):=\int_\Sigma \big( |D^\perp W|^2 -|A^W|^2 \big) \; da -\int_{\partial \Sigma} |W|^2 \; ds \]
where $A^W(\cdot,\cdot):=\langle A(\cdot,\cdot),W\rangle$ denotes the second fundamental form of $\Sigma$ along $W$ and $W$ is any normal variational field along $\Sigma$. 

\begin{definition}
The \emph{Morse index} of $\Sigma$ is defined to be the maximal dimension of a subspace of sections of the normal bundle $N\Sigma$ on which the second variation $\delta^2 \Sigma$ is negative definite. We will denote the Morse index of $\Sigma$ by $\ind(\Sigma)$.
\end{definition}

Intuitively, the Morse index measures the number of independent deformations which decrease area up to second order. As the simplest example, we compute the Morse index of any totally geodesic $\B^k$ in $\B^n$. In this case, $A^W \equiv 0$ and the normal bundle splits isometrically as $\B^k \times \R^{n-k}$. It is clear that $\delta^2 \B^k$ is negative definite along each direction of $\R^{n-k}$. Therefore, we have the following:

\begin{proposition}
The Morse index of the totally geodesic $\B^k$ inside $\B^n$ is equal to $n-k$.
\end{proposition} 

A general Morse index lower bound was proved by Fraser and Schoen in \cite{Fraser-Schoen16}.

\begin{theorem}[Fraser-Schoen \cite{Fraser-Schoen16}]
\label{T:FS-index}
Let $\Sigma^k$ be an immersed free boundary minimal submanifold in $\B^n$ such that $\Sigma$ is not contained in some $\Sigma_0 \times \R$ where $\Sigma_0$ is another $(k-1)$-dimensional immersed free boundary minimal submanifold in $\B^{n-1}$. Then, $\ind(\Sigma) \geq n$.
\end{theorem}

\begin{proof}
Let $v \in \R^n$ be a fixed unit vector. Let $v^\perp$ denote the normal component of $v$ along $\Sigma$. The theorem is proved once we show that
\[ \delta^2 \Sigma(v^\perp,v^\perp)= - k \int_{\Sigma} |v^\perp|^2 \; da.\]
First of all, since $v^\perp$ is a Jacobi field (generated by translation in the $v$-direction), integrating by part we have
\[ \delta^2 \Sigma(v^\perp,v^\perp)=\int_{\partial \Sigma} \big( \langle v^\perp, D_\nu v^\perp \rangle -|v^\perp|^2 \big) \; ds. \]
The next step is to compute the first term in a local frame. Fix $p \in \partial \Sigma$ and a local orthonormal frame $e_1,\cdots,e_k$ of $T\Sigma$ near $p$, where $e_k=\nu=x$ along $\partial \Sigma$. Also, we fix a local orthonormal frame $\nu_1,\cdots, \nu_{n-k}$ of $N\Sigma$ near $p$ such that $D^\perp \nu_\alpha(p)=0$. Let $h_{ij}^\alpha:=\langle A(e_i,e_j),\nu_\alpha\rangle$ be the second fundamental form in this basis. Note that for any $i <k$, we have $h_{ik}^\alpha=\langle D_{e_i} x, \nu_\alpha \rangle =\langle e_i,\nu_\alpha\rangle=0$. Therefore, 
\[ D_\nu \nu_\alpha=\sum_{i=1}^k \langle D_\nu \nu_\alpha,e_i \rangle e_i = -h_{kk}^\alpha x.\]
Using this and the minimality of $\Sigma$, we have
\begin{eqnarray*}
\langle v^\perp, D_\nu v^\perp \rangle &=& \langle v^\perp, D_\nu \left( \sum_{\alpha=1}^{n-k} \langle v,\nu_\alpha \rangle \nu_\alpha \right) \rangle \\
&=& - \langle v^\perp, \sum_{\alpha=1}^{n-k} \left(  h_{kk}^\alpha \langle v,x \rangle \nu_\alpha +h_{kk}^\alpha \langle v,\nu_\alpha \rangle x   \right) \rangle \\
&=&\langle v,x \rangle \left(\sum_{\alpha=1}^{n-k} \sum_{i=1}^{k-1} h_{ii}^\alpha  \langle \nu_\alpha,v \rangle \right).
\end{eqnarray*}
Consider the orthogonal decomposition
\[ v = \langle v,x \rangle x +v_1 + v^\perp \]
where $v_1$ is the component of $v$ tangent to $\partial \Sigma$, we compute
\[ -(k-1) \langle v,x \rangle = \Div_{\partial \Sigma}(v_1 + v^\perp)= \Div_{\partial \Sigma} (v_1) - \sum_{\alpha=1}^{n-k} \sum_{i=1}^{k-1} h_{ii}^\alpha  \langle \nu_\alpha,v \rangle. \]
Putting all these together, we obtain
\[ \delta^2 \Sigma(v^\perp,v^\perp)=\int_{\partial \Sigma} \Big( \langle v,x \rangle \Div_{\partial \Sigma} (v_1) + (k-1) \langle v,x \rangle^2   -|v^\perp|^2 \Big) \; ds. \]
Since $v_1$ is tangent to $\partial \Sigma$, we can apply divergence theorem in the first term to get
\begin{eqnarray*}
\delta^2 \Sigma (v^\perp,v^\perp) &=& \int_{\partial \Sigma} \Big( - |v_1|^2+ (k-1) \langle v,x \rangle^2   -|v^\perp|^2 \Big) \; ds. \\
&=& \int_{\partial \Sigma} \Big( -1 + k \langle v,x \rangle^2 \Big) \; ds.
\end{eqnarray*}
Finally, consider the vector field $V=x - k \langle v,x \rangle v$ and using that $\delta \Sigma(V)=0$,  we have 
\[ \delta^2 \Sigma (v^\perp,v^\perp) = -\int_{\partial \Sigma} \langle V, x \rangle \; ds =-\int_\Sigma \Div_\Sigma V \; da=-k \int_\Sigma |v^\perp|^2 \; da. \]
\end{proof}

We now focus on the hypersurface case (i.e. $k=n-1$). If $\Sigma$ is two-sided, we can fix a global unit normal $N$ on $\Sigma$ and thus any normal variation field can be written as $W=u N$ for some $u \in C^\infty(\Sigma)$. Integrating by parts, we can reduce the stability of the second variation form $\delta^2 \Sigma$ to the following eigenvalue problem:
\[ \left\{ \begin{array}{cl}
Ju  =  \lambda u & \text{in $\Sigma$,} \\
\frac{\partial u}{\partial \nu} +u = 0 &\text{on $\partial \Sigma$,}
\end{array} \right. \] 
where $J:=\Delta_\Sigma -|A_\Sigma|^2$ denotes the Jacobi operator of $\Sigma$. By studying the relationship between the eigenvalues of the Jacobi operator $J$ and the eigenvalues of the Laplacian on $1$-forms, Sargent \cite{Sargent17} and Ambrozio-Carlotto-Sharp \cite{Ambrozio-Carlotto-Sharp18} independently obtained new Morse index bounds in terms of the dimension $n$ and the topology of the minimal hypersurface $\Sigma$. In their works, general (mean)-convex Euclidean domains are considered but for simplicity we only state their results for $\B^3$.

\begin{theorem}[Sargent \cite{Sargent17}, Ambrozio-Carlotto-Sharp \cite{Ambrozio-Carlotto-Sharp18}]
Let $\Sigma$ be an orientable immersed free boundary minimal surface in $\B^3$ with genus $g$ and $k$ boundary components. Then 
\[ \ind(\Sigma) \geq  \lfloor  \frac{2g+k+1}{3} \rfloor .\]
\end{theorem}

Although the lower bound above is not sharp for $\B^3$, for example, it only gives $\ind(\mathbb{D}) \geq 0$ and $\ind(\mathbb{K}) \geq 1$. However it has an important implication that the Morse index cannot stay bounded when \emph{either} the genus or the number of boundary components goes to infinity. This says, in particular, that the examples $\Sigma^{FS}_\ell$, $\Sigma_g^{KL}$, $\Sigma_n^{FPZ}$, $\tilde{\Sigma}_n^{FPZ}$, $\Sigma_g^{KW}$ and $\Sigma_g^{K}$ constructed by various methods in Section \ref{S:Existence} all have unbounded Morse index when either the genus or the number of boundary components goes to infinity. See the recent work of Carlotto and Franz \cite{CF} for some related results.

\subsection{Low index solutions}

Precise Morse index control, when combined with general existence theory, often provides significant insight into geometric problems. Classically this idea has been applied to the study of geodesics which gives the celebrated theorems of Bonnet-Myers, Synge and Frankel. For minimal surfaces in higher codimensions, Micallef and Moore \cite{Micallef-Moore88} prove the topological sphere theorem for closed simply connected manifolds with positive isotropic curvature. In a series of beautiful work \cite{Fraser00} \cite{Fraser02} \cite{Fraser07}, Fraser extended many of the ideas to the free boundary setting, for example to study the topology of closed two-convex hypersurfaces in $\R^n$. We encourage the interested readers to consult the excellent survey \cite{Fraser11} on this subject. 

In the past few years, there has been rapid progress in the theory of closed minimal surfaces in $\Sph^3$. In 2012, Marques and Neves \cite{Marques-Neves14} solved the longstanding Willmore conjecture which asserts that the Willmore energy is uniquely (up to conformal diffeormorphisms) minimized by the Clifford torus among all closed surfaces of genus one. Their proof is a beautiful application of the Almgren-Pitts min-max theory. The arguments in \cite{Marques-Neves14} also used crucially a result of Urbano \cite{Urbano90} saying that the only non-totally geodesic closed minimal surface $\Sigma$ in $\Sph^3$ with $\ind(\Sigma) \leq 5$ is (up to isometries of $\Sph^3$) the Clifford torus. Note that the totally geodesic spheres in $\mathbb{S}^3$ are of index one.

In light of these positive results, there has been a lot of recent interest to compute the Morse index of some explicit examples of free boundary minimal surfaces in $\B^3$. It is clear that the equatorial disks in $\B^3$ have Morse index one. The Morse index of the critical catenoid in $\B^3$ was computed independently by Tran \cite{Tran16}, Smith and Zhou \cite{Smith-Zhou19}, and Devyver \cite{Devyver19}.

\begin{theorem}[Tran \cite{Tran16}, Smith-Zhou \cite{Smith-Zhou19}, Devyver \cite{Devyver19}]
The Morse index of the critical catenoid $\mathbb{K}$ in $\B^3$ is equal to $4$.
\end{theorem}

All the proofs use the separation of variables argument to reduce the problem to study an ODE with Robin boundary conditions. Nonetheless, the analysis is much more complicated than computing the Morse index of the Clifford torus. While it is known that the conformal vector fields in $\Sph^3$ are universally area decreasing for all closed minimal surfaces in $\Sph^3$, it is not known whether the same is true for the conformal vector fields in $\B^3$ in the case of free boundary minimal surfaces. We will return to this point again in Section \ref{S:Area}. In view of Urbano's result, it is natural to ask the following question:

\begin{question}
Is the critical catenoid the only (up to rigid motions) immersed free boundary minimal surface in $\B^3$ with Morse index $4$?
\end{question}

Only partial results have been obtained by Tran \cite{Tran16} and Devyver \cite{Devyver19}. They showed that the answer is yes if we assume additionally that the free boundary minimal surfaces is topologically an annulus. Essentially it would force $\sigma_1=1$ and thus the assertion follows from the uniqueness result of Fraser and Schoen in Theorem \ref{T:FS-unique}.

Finally, we would like to mention that Smith, Stern, Tran and Zhou \cite{Smith-Stern-Tran-Zhou17} have studied the Morse index of the higher dimensional critical catenoid $\mathbb{K}^{n-1}$ in $\B^n$ for $n \geq 4$. The exact Morse index was computed numerically up to $n=101$ and they proved the following asymptotic estimate on the Morse index as $n \to \infty$:
\[ \lim_{n \to \infty} \frac{\log \big(\ind(\mathbb{K}^{n-1}) \big)}{\sqrt{n-1} \log \sqrt{n-1}} =1.\]
An interesting feature of this asymptotics is that for $n$ large, the Killing fields associated to the infinitesimal translations and dilations of $\R^n$ (which has total dimension $n+1$) do not account for all the negative eigenvalues of the index form $\delta^2 \Sigma$. This is in contrast with the situation for the complete higher dimensional catenoids in $\R^n$ (which has Morse index one \cite{Tam-Zhou09}) and the Clifford minimal hypersurfaces in $\Sph^n$ (which has Morse index equal to $n+2$ according to Perdomo \cite{Perdomo01}).

\section{Steklov eigenvalue estimates and compactness theorems}
\label{S:Compactness}

In this section, we describe an estimate for the first Steklov eigenvalue of embedded free boundary minimal hypersurfaces in $\B^n$. We then discuss a few of its consequences including a smooth compactness result for embedded free boundary minimal surfaces in $\B^3$. 

We will need a variant of Reilly's formula \cite[Lemma 2.6]{Fraser-Li14} which applies to domains with piecewise smooth boundary.

\begin{proposition}[Reilly's formula]
\label{P:Reilly}
Let $\Omega \subset \R^n$ be a bounded domain with piecewise smooth boundary $\partial \Omega=\cup_{i=1}^k \Sigma_i$. Denote $S=\cup_{i=1}^k \partial \Sigma_i$. Suppose $f \in C^0(\overline{\Omega})$ is a smooth harmonic function away from $S$ and has uniformly bounded $C^3$-norm on $\Omega \setminus S$. Then we have the following inequality
\[ 0 \geq  \sum_{i=1}^k \int_{\Sigma_i} \left[ (-\Delta^{\Sigma_i} f+ H^{\Sigma_i} \frac{\partial f}{\partial n_i}) \frac{\partial f}{\partial n_i} + \langle \nabla^{\Sigma_i} f, \nabla^{\Sigma_i} \frac{\partial f}{\partial n_i} \rangle + A^{\Sigma_i}(\nabla^{\Sigma_i} f, \nabla^{\Sigma_i} f ) \right] \]
where $\Delta^{\Sigma_i}$ and $\nabla^{\Sigma_i}$ are the intrinsic Laplacian and gradient operators on each $\Sigma_i$; $n_i$ is the inward unit normal of $\Sigma_i$ with respect to $\Omega$; $H^{\Sigma_i}$ and $A^{\Sigma_i}$ are respectively the mean curvature and second fundamental form of $\Sigma_i$ in $\Omega$ with respect to $n_i$.
\end{proposition} 

The inequality was obtained the same way as the standard Reilly formula by integrating the Bochner formula, noting that Stokes' theorem still applies with a singular set of codimension two. We remark that the same formula holds if $\Omega$ is replaced by a Riemannian manifold with non-negative Ricci curvature.

As an immediate application of Proposition \ref{P:Reilly}, we established in \cite[Lemma 2.4]{Fraser-Li14} an embedded Frankel property for free boundary minimal hypersurfaces in $\B^n$.

\begin{theorem}[Fraser-Li \cite{Fraser-Li14}]
\label{T:Frankel}
Any two embedded free boundary minimal hypersurfaces in $\B^n$ must intersect.
\end{theorem}

\begin{proof}
The proof is by contradiction. Suppose there exist two disjoint embedded free boundary minimal hypersurfaces $\Sigma_1$ and $\Sigma_2$ in $\B^n$. Let $\Omega$ be the domain bounded by $\Sigma_1$ and $\Sigma_2$. Note that $\Omega$ is a domain with piecewise smooth boundary $\partial \Omega=\Sigma_1 \cup \Sigma_2 \cup \Gamma$ where $\emptyset \neq \Gamma \subset \partial \B^n$. Consider the mixed boundary value problem
\[ \left\{ \begin{array}{cl}
\Delta f =0 & \text{in $\Omega$}\\
f=0 & \text{on $\Sigma_1$}\\
f=1 & \text{on $\Sigma_2$}\\
\partial f/\partial n = 0 & \text{on $\Gamma$},
\end{array} \right. \]
where $n$ is the outward unit normal of $\partial \B^n$, by elliptic theory there exists a solution $f \in C^{0,\alpha}(\Omega)$ which is smooth with uniform $C^3$ bound in $\Omega$ away from the singular set. Therefore, we can apply Reilly's formula in Proposition \ref{P:Reilly} to get
\[ 0 \geq \int_\Gamma |\nabla^\Gamma f|^2.\]
This implies that $f$ must be constant on $\Gamma$, which is a contradiction since $f=0$ on $\Sigma_1$ and $f=1$ on $\Sigma_2$.
\end{proof}

We remark that the embedded Frankel property also holds inside any Riemannian manifold with non-negative Ricci curvature and strictly convex boundary. This is a crucial ingredient in proving the existence of infinitely many embedded free boundary minimal hypersurfaces in such manifolds using a combination of Lusternik-Schnirelmann theory and Almgren-Pitts min-max theory \cite{Li-Zhou16} (see also \cite[Remark 1.8]{Marques-Neves17}).

\subsection{Steklov eigenvalue estimates}

Recall that the coordinate functions of $\R^n$ restricted to a free boundary minimal submanifold $\Sigma^k \subset \B^n$ are Steklov eigenfunctions with Steklov eigenvalue $1$. Therefore, $\sigma_1(\Sigma) \leq 1$. Inspired by Yau's famous conjecture that the first eigenvalue of the Laplacian on any embedded closed minimal hypersurface in $\Sph^n$ is equal to $n-1$, we conjecture the following in the free boundary setting:

\begin{question}
If $\Sigma \subset \B^n$ is an embedded free boundary minimal hypersurface, then $\sigma_1(\Sigma)=1$.
\end{question}

In \cite{Fraser-Li14}, we proved a lower bound of $\sigma_1(\Sigma)$ which is similar to the bound obtained by Choi and Wang \cite{Choi-Wang83} for closed embedded minimal hypersurface in $\Sph^n$.

\begin{theorem}[Fraser-Li \cite{Fraser-Li14}]
\label{T:FL-bound}
Let $\Sigma$ be an embedded free boundary minimal hypersurface in $\B^n$. Then $\sigma_1(\Sigma) \geq 1/2$.
\end{theorem} 

\begin{proof}
Let $\Omega_1$ and $\Omega_2$ be the two connected components of $\B^n \setminus \Sigma$. Take $\Omega=\Omega_1$ and $\partial \Omega=\Sigma \cup \Gamma$ where $\Gamma \subset \partial \B^n$. Note that $\Gamma$ may not be connected. Let $u \in C^\infty(\Sigma)$ be a first Steklov eigenfunction of $\Sigma$. Since $\partial \Gamma=\partial \Sigma$, we can extend $u|_{\partial \Gamma}$ into $\Gamma$ harmonically. Finally, we can take the extended function $u$ on $\partial \Omega=\Sigma \cup \Gamma$ and extend once again harmonically into $\Omega$ to obtain a function $f$ in $\Omega$. By elliptic theory one can see that $f \in C^{1,\alpha}(\Omega) \cap C^\infty (\Sigma \setminus \partial \Sigma)$ with uniform $C^3$ bound away from the singular set $\partial \Sigma$. Plugging this harmonic function $f$ into Reilly's formula in Proposition \ref{P:Reilly}, we obtain
\[ 0 \geq \int_\Sigma \left( \langle \nabla^\Sigma f,\nabla \frac{\partial f}{\partial n_\Sigma} \rangle + A^\Sigma(\nabla^\Sigma f, \nabla^\Sigma f) \right) + \int_\Gamma \left( \langle \nabla^\Gamma f,\nabla^\Gamma \frac{\partial f}{\partial n_\Gamma} \rangle + |\nabla^\Gamma f|^2 \right), \]
where $n_\Sigma$ and $n_\Gamma$ are the inward (with respect to $\Omega$) unit normals of $\Sigma$ and $\Gamma$ respectively. By choosing $\Omega=\Omega_2$ if necessary, we can assume that $\int_\Sigma A^\Sigma(\nabla^\Sigma f, \nabla^\Sigma f) \geq 0$. Integrating by parts, we obtain
\[ 0 \geq \int_{\partial \Sigma} \frac{\partial f}{\partial \nu_\Sigma} \frac{\partial f}{\partial n_\Sigma} + \int_{\partial \Gamma} \frac{\partial f}{\partial \nu_\Gamma} \frac{\partial f}{\partial n_\Gamma} + \int_\Gamma |\nabla^\Gamma f|^2, \]
where $\nu_\Sigma$ and $\nu_\Gamma$ are the outward conormal vectors of $\partial \Sigma=\partial \Gamma$ with respect to $\Sigma$ and $\Gamma$ respectively. Using the free boundary condition, we have along $\partial \Sigma=\partial \Gamma$,
\[ \nu_\Sigma= -n_\Gamma  \qquad \text{ and } \qquad n_\Sigma=-\nu_\Gamma.\]
Therefore, we can rewrite the above inequality as
\[ 0 \geq -2 \sigma_1(\Sigma) \int_{\partial \Gamma} f \frac{\partial f}{\partial \nu_\Gamma} + \int_\Gamma |\nabla^\Gamma f|^2.\]
Integrating by part gives $\int_{\partial \Gamma} f \frac{\partial f}{\partial \nu_\Gamma}=\int_\Gamma |\nabla^\Gamma f|^2$. Thus, we obtain the inequality $\sigma_1(\Sigma) \geq 1/2$. Note that $f$ is non-constant on $\Gamma$ since $f|_{\partial \Sigma}=u$ is non-constant.
\end{proof}

With a bit more work, one can in fact show that we have the strict inequality $\sigma_1(\Sigma) >1/2$ (see \cite{Batista-Cunha16}). Nonetheless, our lower bound is sufficient to imply a smooth compactness theorem for embedded free boundary minimal surfaces in $\B^3$ to be described in the next section.

Although we are still far from having a definite answer to \textbf{Open question 7}, it has been verified that some of the known examples constructed in Section \ref{S:Existence} satisfy $\sigma_1(\Sigma)=1$. It is obvious for the Fraser-Schoen surfaces $\Sigma^{FS}_\ell$ since they are constructed by embeddings using the first Steklov eigenfunctions. It is easy to check by direct computation that all the higher dimensional critical catenoids have $\sigma_1=1$. In a recent work of McGrath \cite{McGrath18}, it was proved that $\sigma_1=1$ for any embedded free boundary minimal surface in $\B^3$ which is invariant under certain groups of reflections such that their fundamental domains satisfy some additional assumptions. In particular, this implies that the Folha-Pacard-Zolotareva surfaces $\Sigma_n^{FPZ}$ and $\tilde{\Sigma}_n^{FPZ}$ all have $\sigma_1=1$. McGrath's proof uses the nodal domain theorem for Steklov eigenfunctions and a symmetrization technique due to Choe and Soret \cite{Choe-Soret09}. It would be interesting to check whether the other known examples satisfy $\sigma_1=1$. See \cite{GL} for some related results by Girouard and Lagace.

\subsection{Smooth compactness results}

We now turn to the study of the space of all free boundary minimal hypersurfaces in $\B^n$. In particular, we are interested in what natural conditions would imply that the space is compact, under suitable topology. Along this direction, we obtained the first compactness result for embedded free boundary minimal surfaces in $\B^3$ with a fixed topological type. Our result is analogous to the celebrated compactness theorem of Choi and Schoen \cite{Choi-Schoen85} for embedded minimal surfaces in $\Sph^3$.

\begin{theorem}[Fraser-Li \cite{Fraser-Li14}]
Given any integers $g \geq 0$ and $k \geq 1$, the space of all embedded free boundary minimal surfaces in $\B^3$ with genus $g$ and $k$ boundary components is compact in the $C^\infty$-topology. In other words, any sequence $\{\Sigma_i\}_{i \in \N}$ of such minimal surfaces has a subsequence converging smoothly as a graph to a smooth embedded free boundary minimal surface $\Sigma_\infty$ in $\B^3$ of the same topological type.
\end{theorem}

\begin{proof}
We will briefly describe the main steps of the proof. First of all, the coarse upper bound in Proposition \ref{P:coarse-bound} and our Steklov eigenvalue lower bound in Theorem \ref{T:FL-bound} together imply a uniform bound on the boundary length $|\partial \Sigma| \leq 4\pi (g+k)$. By Proposition \ref{P:length-area} we also have a uniform area bound $|\Sigma| \leq 2 \pi(g+k)$. Using minimality of $\Sigma$, the Gauss equation and Gauss-Bonnet theorem, we have the following uniform $L^2$-bound on the second fundamental form:
\[ \int_\Sigma |A^\Sigma|^2 \; da= \int_\Sigma -2K^\Sigma \; da=-2 \left( 2 \pi \chi(\Sigma) -\int_{\partial \Sigma} k_g \; ds \right) \leq 4 \pi (4g+3k-2).\]
Here, the free boundary condition implies that the geodesic curvature $k_g$ of $\partial \Sigma$ inside $\Sigma$ is equal to $1$. By a general compactness result we know that any sequence $\{\Sigma_i\}_{i \in \N}$ of embedded free boundary minimal surfaces in $\B^3$ with genus $g$ and $k$ boundary components would have a subsequence converging (as a multi-sheeted graph) away from finitely many points $p_1,\cdots, p_N$ to a limit embedded free boundary minimal surface $\Sigma_\infty$ in $\B^3$. By a removable singularity theorem \cite[Theorem 4.1]{Fraser-Li14}, $\Sigma_\infty$ is smooth across each $p_1,\cdots, p_N$. Moreover, by a logarithmic cut-off trick together with the uniform bound $\sigma_1 \geq 1/2$ we see that the multiplicity is one.
\end{proof}

In higher dimensions, the control on topology is not enough to guarantee a smooth compactness theorem. For example, the $O(2)\times O(2)$ symmetric embedded examples $\Sigma_{2,2}(\ell)$ in $\B^4$ constructed by Freidin, Gulian and McGrath (see Theorem \ref{T:low-cohomo}) are all topologically $\B^2 \times \Sph^1$ but as $\ell \to \infty$ they converge to the minimal cone (which is singular at the origin) over the Clifford torus in $\Sph^3$. In \cite{Ambrozio-Carlotto-Sharp18}, Ambrozio, Carlotto and Sharp proved the following compactness result in higher dimensions (but still in codimension one):

\begin{theorem}[Ambrozio-Carlotto-Sharp \cite{Ambrozio-Carlotto-Sharp18}]
\label{T:ACS}
Given any integer $I \geq 0$ and positive constant $\Lambda>0$, the space of all embedded free boundary minimal hypersurfaces in $\B^n$ with area at most $\Lambda$ and Morse index at most $I$ is compact in the $C^\infty$-topology.
\end{theorem}

In particular, Theorem \ref{T:ACS} implies that the family $\{\Sigma^{FGM}_{k_1,k_2}(\ell)\}_{\ell \in \N}$ of embedded free boundary minimal hypersurfaces in $\B^{k_1+k_2}$ constructed by Freidin-Gulian-McGrath in Theorem \ref{T:low-cohomo} has unbounded Morse index as $\ell \to \infty$. In fact, it was shown in \cite{Ambrozio-Carlotto-Sharp18} that the same smooth compactness theorem holds when one replaces the Morse index bound by a uniform lower bound on the $p$-th eigenvalue of the stability operator for some $p \in \N$. Such compactness results were established earlier by Sharp \cite{Sharp17} and Ambrozio-Carlotto-Sharp \cite{Ambrozio-Carlotto-Sharp16} for closed embedded minimal hypersurfaces in $\Sph^n$.

\section{Area bounds}
\label{S:Area}

In this final section, we study the area of free boundary minimal submanifolds in $\B^n$. It is a well-known result by Li and Yau \cite{Li-Yau82} that minimal surfaces in $\Sph^n$ maximize area in their conformal orbit. This has been generalized to higher dimensional minimal submanifolds in $\Sph^n$ by El Soufi and Ilias \cite{ElSoufi-Ilias86}. One is then naturally led to the following question in the free boundary setting:

\begin{question}
Let $\Sigma^k$ be an immersed free boundary minimal submanifold in $\B^n$, and $f:\B^n \to \B^n$ be a conformal diffeomorphism. Is it true that $|f(\Sigma)| \leq |\Sigma|$?
\end{question}

Very limited results are known. For $k=2$, Fraser and Schoen \cite{Fraser-Schoen11} showed that the boundary of any free boundary minimal surface in $\B^n$ maximizes length in its conformal orbit.

\begin{theorem}
\label{T:FS-conformal}
Let $\Sigma$ be an immersed free boundary minimal surface in $\B^n$. Suppose $f:\B^n \to \B^n$ is a conformal diffeomorphism. Then we have $|f(\partial \Sigma)| \leq |\partial \Sigma|$.
\end{theorem}

\begin{proof}
The original proof in \cite{Fraser-Schoen11} uses the conformal invariance of the trace-free second fundamental form of $\Sigma$. We recall here an alternative proof taken from \cite{Fraser-Schoen13} using a first variation argument. Using our understanding of the conformal group of $\B^n$, there exists some $y \in \R^n$ with $|y| >1$ such that
\[ f^*g_{\text{Eucl}}=\left(\frac{|y|^2-1}{|x-y|^2}\right)^2g_{\text{Eucl}},\]
where $g_{\text{Eucl}}$ is the Euclidean metric on $\B^n$. Define a vector field $X=(x-y)/|x-y|^2$ on $\B^n$. Note that $X$ is not tangential to $\partial \B^n$. However, using that $\Sigma$ is a free boundary minimal surface, the divergence theorem on $\Sigma$ implies
\[ \int_\Sigma \Div_\Sigma X \; da =\int_{\partial \Sigma} \langle X,x \rangle \; ds.\]
By a direct computation, we have 
\[ \Div_\Sigma X = \frac{2}{|x-y|^2} -\frac{2 |(x-y)^T|^2}{|x-y|^4} \geq 0,\] 
\[ \langle X, x \rangle = \frac{1}{2} \left( 1- \frac{|y|^2-1}{|x-y|^2} \right).\]
Plugging these into the equality above, using the fact that $\partial \Sigma$ is one-dimensional, 
\[ 0 \leq \frac{1}{2} \int_{\partial \Sigma} \left( 1- \frac{|y|^2-1}{|x-y|^2} \right) \; ds =\frac{1}{2} \big(|\partial \Sigma| -|f(\partial \Sigma)|\big).\]
\end{proof}

By a conformal blow-up at any boundary point $p \in \partial \Sigma$, we have the following sharp area lower bound for free boundary minimal surfaces in $\B^n$. Note that $2 |\Sigma|=|\partial \Sigma|$ by Proposition \ref{P:length-area}.

\begin{theorem}[Fraser-Schoen \cite{Fraser-Schoen11}]
Let $\Sigma^2$ be an immersed free boundary minimal surface in $\B^n$ Then, 
\[ |\Sigma^2| \geq |\B^2|=\pi \]
and equality holds if and only if $\Sigma^2$ is an equatorial plane disk.
\end{theorem} 

By a direct monotonicity-type argument, Brendle was abled to generalize Fraser and Schoen's result to any dimension and codimension.

\begin{theorem}[Brendle \cite{Brendle12}]
\label{T:Brendle}
Let $\Sigma^k$ be an immersed free boundary minimal submanifold in $\B^n$ Then, 
\[ |\Sigma^k| \geq |\B^k| \]
and equality holds if and only if $\Sigma^k$ is congruent to $\B^k$.
\end{theorem}

Note that the sharp area bounds above imply the sharp isoperimetric inequality for free boundary minimal submanifolds in $\B^n$ (see \cite[Corollary 5.5]{Fraser-Schoen11} and \cite[Corollary 5]{Brendle12}). Motivated by Theorem \ref{T:FS-conformal}, Fraser and Schoen conjectured that the same statement hold in arbitrary dimension and codimension.

\begin{question}
Let $\Sigma^k$ be an immersed free boundary minimal submanifold in $\B^n$. Suppose $f:\B^n \to \B^n$ is a conformal diffeomorphism. Is it true that $|f(\partial \Sigma)| \leq |\partial \Sigma|$?
\end{question}

Note that an affirmative answer to \textbf{Open Question 9} would imply Theorem \ref{T:Brendle}. When $\Sigma$ is a cone, it follows from the corresponding results of Li-Yau and El Soufi-Ilias applied to its link in $\Sph^{n-1}$. Moreover, Fraser and Schoen \cite[Theorem 3.8]{Fraser-Schoen13} proved that the boundary volume of $\partial \Sigma$ decreases up to second order under conformal changes. On the other hand, it is not known that the volume of $\Sigma$ itself decreases up to second order under conformal changes. However, it can be checked numerically that the critical catenoid does satisfy the conclusion of \textbf{Open Question 8}.

After knowing that $\B^k$ gives the sharp lower bound for free boundary minimal submanifolds in $\B^n$, it is natural to ask what is the next smallest area, at least when $k=2=n-1$. In the proof of Willmore conjecture by Marques and Neves \cite{Marques-Neves14}, a key step in their argument is to show that the Clifford torus in $\Sph^3$ is the (unique) closed minimal surface with smallest area other than the geodesic spheres. In view of this spectacular result, we conjecture the following:

\begin{question}
Let $\Sigma^2$ be an immersed free boundary minimal surface in $\B^3$ which is not the equatorial disk. Is it true that $|\Sigma| \geq |\mathbb{K}|$? 
\end{question}

The question above should be closely related to the sharp Morse index lower bound conjectured in \textbf{Open Question 6}.

\bibliographystyle{amsplain}
\bibliography{references}

\end{document}